\documentclass[11pt,a4paper,reqno]{amsart}
\usepackage[T1]{fontenc}
\usepackage{lmodern}
\usepackage{amsmath,amssymb,mathtools}
\usepackage{array,booktabs,longtable,tabularx}
\usepackage[margin=26mm]{geometry}
\usepackage{microtype}
\usepackage{url}
\usepackage[hidelinks]{hyperref}
\hypersetup{pdftitle={Automorphisms of triality groups over local rings},
pdfauthor={Elena I. Bunina, Avi Berezovskiy, Deep H. Makadiya}}
\newtheorem{theorem}{Theorem}[section]
\newtheorem{proposition}[theorem]{Proposition}
\newtheorem{lemma}[theorem]{Lemma}

\theoremstyle{definition}
\newtheorem{definition}[theorem]{Definition}

\newtheorem{example}[theorem]{Example}
\newcommand{\Cent}{\operatorname{Cent}}

\newcommand{\Aut}{\operatorname{Aut}}
\newcommand{\Int}{\operatorname{Int}}

\newcommand{\Spec}{\operatorname{Spec}}
\newcommand{\GL}{\operatorname{GL}}
\newcommand{\SL}{\operatorname{SL}}
\newcommand{\Jac}{\operatorname{Jac}}

\newcommand{\Scover}{\mathcal S}
\newcommand{\Llift}{\widetilde L}
\newcommand{\U}{\widetilde U}
\newcommand{\id}{\operatorname{id}}
\newcommand{\commL}[2]{[#1,#2]_{\!L}}
\newcommand{\commR}[2]{[#1,#2]_{\!R}}
\theoremstyle{plain}
\newtheorem{maintheorem}{Theorem}

\numberwithin{equation}{section}
\title[Automorphisms of triality groups]{Automorphisms of triality groups over local rings}
\author{Elena I. Bunina}
\author{Avi Berezovskiy}
\author{Deep H. Makadiya}
\address{Department of Mathematics, Bar-Ilan University, Ramat Gan, Israel}
\date{}
\begin{document}

\begin{abstract}
Let $R$ be a nonzero commutative local ring with $2,3\in R^\times$,
and let $\rho\in\Aut(R)$ satisfy $\rho^3=\id$.
We classify abstract automorphisms of the elementary and full triality
groups in the adjoint and simply connected forms of type $D_4$.
Every automorphism is uniquely determined by an adjoint group element
and a ring automorphism preserving the trace $T=\id+\rho+\rho^2$;
conversely, every such pair defines an automorphism.
A square-zero example shows that trace-preserving coefficient
automorphisms need not normalize $\langle\rho\rangle$ when $\rho$
is nontrivial on $R$ but trivial on its residue field.
We also prove a relative rigidity theorem for homomorphisms from
elementary triality groups over arbitrary coefficient rings to local
targets with $2$ and $3$ invertible, assuming that an adjoint
conjugation makes the parameter-one root images standard modulo
the maximal ideal. Conversely, every unital trace-preserving ring
homomorphism into such a local coefficient ring induces the
corresponding homomorphism of elementary groups.
\end{abstract}

\maketitle

%-------------------

\section{Introduction and main results}\label{sec:introduction}

We classify abstract automorphisms of triality groups over commutative
local rings with $2$ and $3$ invertible.
The coefficient action is governed by the trace $T=\id+\rho+\rho^2$
of the defining ring automorphism $\rho$, rather than by preservation
of $\rho$ up to inversion. When a nontrivial $\rho$ induces the
identity on the residue field, additional trace-preserving coefficient
symmetries can occur.

Over fields, the classical theorems of Steinberg and Humphreys describe
automorphisms by inner, diagonal, field, and graph factors
\cite{Steinberg1960,Steinberg,Humphreys1969}.
Over local rings, one must also prove that the reduction kernel is
characteristic and recover the action on root parameters over the ring.
The structural and normal-subgroup results of Abe and Suzuki provide
the background
\cite{Abe1969,AbeSuzuki1976,Abe1977,Suzuki1977,Suzuki1995}.
Automorphisms of untwisted groups over local and more general
commutative rings are treated in
\cite{BuninaADE2009,BuninaG2,Bunina2012,Bunina2024};
the twisted normal-structure results of Garge and Makadiya \cite{GM}
provide further input.

This paper completes the local triality case in the programme covering
${}^2A_\ell$, $\ell\geq5$, the low-rank types ${}^2A_3,{}^2A_4$,
${}^2D_\ell$, $\ell\geq4$, and ${}^2E_6$
\cite{BM-A-high,BM-A-low,BMlocalD,BMlocalE}.
Unlike the stable order-two cases, the coefficient actions here need
not commute with the defining ring automorphism.
We also prove a relative theorem with an arbitrary source ring and
a local target: composing an automorphism with localization gives
such a homomorphism, not necessarily a local automorphism.
This supplies the triality input for the separate local--global
programme of \cite{BM-global}; no global classification is used here.

\subsection{The groups and the classification}

For the classification, let $R$ be a nonzero commutative local ring
with $2,3\in R^\times$, and let $\rho\in\Aut(R)$ satisfy
$\rho^3=\id$.
Put $J=\Jac(R)$, $k=R/J$, and $R_0=R^\rho$, and write $\bar\rho$
for the induced automorphism of $k$.
We allow $\rho=\id$, as well as $\rho\ne\id$ with $\bar\rho=\id$.
The ring need not be Noetherian, reduced, complete, or Henselian,
and $R/R_0$ need not be cubic \'etale.

Let $\mathbf G_Q$ and $\mathbf G_P$ be the split adjoint and simply
connected group schemes of type $D_4$, with character lattices
the root lattice $Q$ and the weight lattice $P$, respectively.
Fix the pinned triality automorphism $\tau$ cycling
$\alpha_1\mapsto\alpha_3\mapsto\alpha_4\mapsto\alpha_1$
and fixing $\alpha_2$, and put
\[
 G_\Lambda(R)=\mathbf G_\Lambda(R)^{\tau\rho},
 \qquad \Lambda=Q,P.
\]
Their elementary subgroups $E_\Lambda(R)$ are generated by the relative
root elements. The relative root system is $G_2$, with short simple
root $b$ and long simple root $a$, whose root maps are
\[
 \begin{aligned}
 y_b(t)&=x_{\alpha_4}(t)x_{\alpha_1}(\rho(t))
         x_{\alpha_3}(\rho^2(t)), &&t\in R,\\
 y_a(v)&=x_{\alpha_2}(v), &&v\in R_0.
 \end{aligned}
\]
Section~\ref{sec:coordinates} fixes the signed root maps so that
Weyl transport changes parameters only by a sign.
Write $R_\gamma=R$ for short roots and $R_\gamma=R_0$ for long roots.
The indexed family $\mathcal F_R=(y_\gamma(1))_{\gamma\in G_2}$
is the \emph{constant frame}.

Define
\[
 \Aut_T(R,\rho)=\{\eta\in\Aut(R):\eta\circ T=T\circ\eta\}.
\]
For $\eta\in\Aut_T(R,\rho)$, the coefficient action is prescribed by
$\phi_\eta(y_\gamma(t))=y_\gamma(\eta(t))$ for $t\in R_\gamma$.
Its existence on the elementary group, not merely on the group
presented by the nonopposite root relations, is part of the theorem.

\begin{maintheorem}\label{thm:main-A}\label{thm:four-groups}
For every $\eta\in\Aut_T(R,\rho)$, the root rule defines an
automorphism $\phi_\eta$ of $E_Q(R)$, which extends uniquely to
an automorphism $\Phi_\eta$ of $G_Q(R)$.
Every automorphism of either group has a unique expression
\[
 \begin{aligned}
 \alpha&=\Int(g)\circ\phi_\eta &&\text{on }E_Q(R),\\
 \alpha&=\Int(g)\circ\Phi_\eta &&\text{on }G_Q(R),
 \end{aligned}
\]
where $g\in G_Q(R)$ and $\eta\in\Aut_T(R,\rho)$.

The central isogeny induces an isomorphism
$\pi_E:E_P(R)\xrightarrow{\sim}E_Q(R)$, and $G_P(R)=E_P(R)$.
Put $\phi_{\eta,P}=\pi_E^{-1}\circ\phi_\eta\circ\pi_E$.
Every automorphism of $E_P(R)=G_P(R)$ has a unique expression
\[
 \alpha=\operatorname{Ad}_P(g)\circ\phi_{\eta,P},
 \qquad g\in G_Q(R),\quad \eta\in\Aut_T(R,\rho).
\]
Conversely, every such pair $(g,\eta)$ defines an automorphism
of each of the four groups.
\end{maintheorem}

Here $\Int(g)(x)=gxg^{-1}$, and $\operatorname{Ad}_P$ is the
canonical adjoint action on the simply connected form.
It requires no lift of $g$ to $G_P(R)$; such a lift can fail to exist
even when $g$ reduces to the identity (Example~\ref{ex:torus}).
The adjoint factor combines inner and diagonal actions, whose further
separation need not be unique. All four groups are centerless,
so no central factor occurs.

If $\rho=\id$, the groups and the theorem reduce to split $G_2$.
If $\bar\rho$ has order three, Proposition~\ref{prop:local-orientation}
gives
\[
 \Aut_T(R,\rho)=N_{\Aut(R)}(\langle\rho\rangle).
\]
Thus $\eta\rho=\rho\eta$ or $\eta\rho=\rho^{-1}\eta$ on the whole
local ring. This can fail when $\rho\ne\id$ but $\bar\rho=\id$.
In Example~\ref{ex:squarezero}, the square-zero algebra
$R=\mathbb Q\oplus V$, $V=\mathbb Q^2$, admits a nontrivial $\rho$
with $T(q,v)=(3q,0)$. Every $L\in\GL_2(\mathbb Q)$ therefore gives
a trace-preserving ring automorphism $(q,v)\mapsto(q,Lv)$,
but not all normalize $\langle\rho\rangle$.
Section~\ref{sec:examples} constructs these group actions directly.
After fixing the constant frame, one recovers the ring and its trace,
but not necessarily $\rho$ up to inversion.

\subsection{Relative rigidity and the proof}

The relative theorem concerns homomorphisms from arbitrary coefficient
rings to local targets.
Let $(S,\rho_S)$ be a coefficient pair: $S$ is a nonzero commutative
ring with identity and $\rho_S\in\Aut(S)$ satisfies $\rho_S^3=\id$.
Let $(A,\rho_A)$ satisfy the local hypotheses above.
Put $E_S=E_Q(S)$, $E_A=E_Q(A)$, and $k=A/\Jac(A)$, and write
$T_S,T_A$ for the corresponding traces.
As before, $S_\gamma=S$ for short roots and $S_\gamma=S^{\rho_S}$
for long roots.

A homomorphism $\Psi:E_S\to E_A$ has a
\emph{liftable prepared residual frame} if some $h\in G_Q(A)$ satisfies
\[
 \operatorname{red}_A
 \bigl(h^{-1}\Psi(y_{\gamma,S}(1))h\bigr)=y_{\gamma,k}(1)
 \qquad(\gamma\in G_2).
\]
It suffices to check $\gamma\in\{\pm a,\pm b\}$.

\begin{maintheorem}\label{thm:main-B}
Every homomorphism $\Psi:E_S\to E_A$ with a liftable prepared
residual frame admits a unique pair $(g,\eta)$, where $g\in G_Q(A)$
and $\eta:S\to A$ is a unital ring homomorphism, such that
\[
 \eta\circ T_S=T_A\circ\eta,\qquad
 \eta(S^{\rho_S})\subseteq A^{\rho_A},
\]
and
\[
 \Psi(y_{\gamma,S}(t))=g\,y_{\gamma,A}(\eta(t))g^{-1}
 \qquad(\gamma\in G_2,\ t\in S_\gamma).
\]
For any preparation $h$, one has $\bar g=\bar h$.

Conversely, every unital trace-preserving ring homomorphism
$\eta:S\to A$ induces a unique homomorphism $E_S\to E_A$
by the same root rule with $g=1$.
Neither injectivity nor surjectivity is assumed, and the source
need not contain $1/2$ or $1/3$.
\end{maintheorem}

The proof has three parts.
First, after preparing the residual frame, we normalize the Weyl
representatives by averaging over the relative Tits group of order $48$.
The long simple-root image is treated by elimination; for the short
root, an explicit centralizer orbit supplies a comparison solution
and an invertible Jacobian minor proves local uniqueness.
A determinant identity corrects the resulting change of basis to
a proper orthogonal similitude. Triviality of the adjoint frame
centralizer forces its projective class to be fixed by $\tau\rho_A$.

Second, commutation with the frame recognizes the root subgroups,
and Weyl transport and commutator relations recover a single unital
trace-preserving ring map.
For the converse, a trace-preserving map acts on the graded covers
of \cite{EJK}. Using rank-one lifting, a congruence decomposition,
and the field comparison of \cite{Stavrova}, we prove that the local
cover kernel is its center.
Centrality of the source kernel and triviality of the target frame
centralizer give coefficient substitutions for local sources;
localization gives the converse for arbitrary sources and local targets.

Third, independently of cover centrality, we show that
$\ker(E_Q(R)\to E_Q(k))$ is the greatest proper normal subgroup
of $E_Q(R)$, and hence characteristic.
The automorphism theorem over $k$ and lifting of the adjoint
conjugator supply the residual frame required by the relative theorem.
Torus compatibility and characteristicity of the elementary subgroup
give the full adjoint classification; the central isogeny gives
the simply connected form.

Section~\ref{sec:coordinates} fixes the root formulas;
Sections~\ref{sec:normalization} and~\ref{sec:recognition} prove
normalization and relative rigidity.
Section~\ref{sec:centrality} constructs the coefficient actions,
and Section~\ref{sec:classification} proves Theorem~\ref{thm:main-A}.
Section~\ref{sec:examples} treats coefficient symmetries and examples.
The appendices contain the finite normalization data and the
root-factor extraction proof. 
The accompanying computational supplement~\cite{TrialitySupplement}
records the expanded polynomials and the square-zero calculation.

%---------------------
%---------------------
%---------------------

\medskip
\section{Triality groups and root coordinates}
\label{sec:coordinates}\label{sec:d4-8d-vector-v4}
\label{sec:d4-weyl-coherent-parameters}

We fix the root coordinates and establish the local lifting and
factorization properties used below.
A \emph{coefficient pair} consists of a nonzero commutative ring $R$
with identity and an automorphism $\rho$ satisfying $\rho^3=\id$.
All ring maps are unital. Put
\[
 R_0=R^\rho,\qquad T(t)=t+\rho(t)+\rho^2(t),\qquad
 t^\#=\rho(t)\rho^2(t),\qquad N(t)=tt^\#,
\]
\[
 t\times u=(t+u)^\#-t^\#-u^\#
          =\rho(t)\rho^2(u)+\rho(u)\rho^2(t).
\]
The maps $T,N$ take values in $R_0$, $T$ is $R_0$-linear, and
$\times$ is symmetric and $R_0$-bilinear.
A pair is \emph{local} when $R$ is local and $2,3\in R^\times$;
then write $J=\Jac(R)$, $k=R/J$, $J_0=J\cap R_0$, and
$\bar\rho$ for the induced automorphism of $k$.

For the split pinned forms $\mathbf G_Q,\mathbf G_P$ of
Section~\ref{sec:introduction}, choose the simple roots
\[
 \alpha_1=e_1-e_2,\qquad \alpha_2=e_2-e_3,\qquad
 \alpha_3=e_3-e_4,\qquad \alpha_4=e_3+e_4.
\]
The pinned triality automorphism $\tau$ cycles
$\alpha_1\mapsto\alpha_3\mapsto\alpha_4\mapsto\alpha_1$
and fixes $\alpha_2$. It commutes with the coefficient action of
$\rho$. Set
\[
 \delta=\tau\rho,\qquad
 G_\Lambda(R)=\mathbf G_\Lambda(R)^\delta
 \qquad(\Lambda=Q,P).
\]
The subgroup $E_\Lambda(R)$ is generated by the relative root elements
below; abbreviate $E_R=E_Q(R)$ and $G_R=G_Q(R)$.
Triality cycles the three nonzero classes of
$P/Q\simeq(\mathbb Z/2)^2$, so $Q,P$ are the only
triality-invariant isogeny types.

\subsection{The matrix model}

We use the vector representation to fix the absolute root signs.
In the ordered basis
$(e_1,e_2,e_3,e_4,e_{-1},e_{-2},e_{-3},e_{-4})$, let
\[
 q_8(x)=\sum_{i=1}^4x_ix_{-i},\qquad
 J_8=\begin{pmatrix}0&I_4\\ I_4&0\end{pmatrix},
\]
where $J_8$ is the polar matrix. For $1\leq i<j\leq4$, set
\begin{align*}
 M_{e_i-e_j}&=E_{i,j}-E_{-j,-i},&
 M_{-e_i+e_j}&=E_{j,i}-E_{-i,-j},\\
 M_{e_i+e_j}&=E_{i,-j}-E_{j,-i},&
 M_{-e_i-e_j}&=-E_{-i,j}+E_{-j,i}.
\end{align*}
Then $M_{-r}=M_r^t$, $M_r^2=0$, and $X_r(t)=I_8+tM_r$
preserves $q_8$ and has determinant one over every coefficient ring.
The central isogenies
$\mathbf G_P\to\mathbf{SO}_8\to\mathbf G_Q$ send
$x_{r,P}(t)$ to $X_r(t)$ and then to $x_{r,Q}(t)$.
Root words are read in either pinned form or in these matrices.
Triality acts on the pinned groups, not by conjugation with an
$8\times8$ matrix.

For a local pair, $\operatorname{Pic}(R)=0$, so
\[
 \mathbf G_Q(R)=\operatorname{PGO}_8^+(R)
 \simeq\operatorname{GO}_8^+(R)/R^\times,
\]
where the quotient is by scalar matrices and
\[
 \operatorname{GO}_8^+(R)=
 \{C\in\GL_8(R):C^tJ_8C=\lambda J_8,\quad
   \det C=\lambda^4\text{ for some }\lambda\in R^\times\}.
\]
This adjoint group of points, also denoted by $\operatorname{PSO}_8(R)$,
need not equal $\operatorname{SO}_8(R,J_8)/\{\pm I_8\}$:
the multiplier need not be a square.
Every element of $E_R$ nevertheless has an orthogonal determinant-one
representative, obtained from a root word.
Over source rings without $1/2$, the integral root maps are interpreted
in the split group scheme of $q_8$.

\subsection{Relative root maps}

We choose parameters so that signed Weyl transport involves no powers
of $\rho$. Write
\[
 b=\beta_s,\qquad a=\beta_l,\qquad
 c=b+a,\qquad d=2b+a,\qquad e=3b+a,\qquad f=3b+2a.
\]
The roots $b,c,d$ are short and $a,e,f$ are long.
Put $R_\gamma=R$ for short roots and $R_\gamma=R_0$ for long roots.
Suppressing the isogeny-type subscript, their matrix representatives are
\begin{align*}
 y_{\epsilon b}(t)
 &=X_{\epsilon\alpha_4}(t)
   X_{\epsilon\alpha_1}(\rho(t))
   X_{\epsilon\alpha_3}(\rho^2(t)),\\
 y_{\epsilon a}(v)
 &=X_{\epsilon\alpha_2}(v),\\
 y_{\epsilon c}(t)
 &=X_{\epsilon(\alpha_2+\alpha_4)}(t)
   X_{\epsilon(\alpha_1+\alpha_2)}(-\rho(t))
   X_{\epsilon(\alpha_2+\alpha_3)}(\rho^2(t)),\\
 y_{\epsilon d}(t)
 &=X_{\epsilon(\alpha_2+\alpha_3+\alpha_4)}(\rho(t))
   X_{\epsilon(\alpha_1+\alpha_2+\alpha_4)}(\rho^2(t))
   X_{\epsilon(\alpha_1+\alpha_2+\alpha_3)}(t),\\
 y_{\epsilon e}(v)
 &=X_{\epsilon(\alpha_1+\alpha_2+\alpha_3+\alpha_4)}(v),\\
 y_{\epsilon f}(v)
 &=X_{\epsilon(\alpha_1+2\alpha_2+\alpha_3+\alpha_4)}(v),
 \qquad \epsilon\in\{1,-1\},\quad t\in R,\quad v\in R_0.
\end{align*}
In the pinned forms, replace $X_r$ by $x_{r,\Lambda}$.
The orbit factors commute. The cyclic signs of $\tau$ on the
$c$-orbit, in the displayed order, are $-1,-1,+1$;
on the other orbits they are positive.
Thus all products are $\delta$-fixed.
Negative root maps retain the same signs and powers of $\rho$,
so $y_{-\gamma}(t)=y_\gamma(t)^t$ in matrices.
Absolute unipotent coordinates make every root map injective, and
\[
 y_\gamma(t)y_\gamma(u)=y_\gamma(t+u)
 \qquad(t,u\in R_\gamma).
\]

Let $s_s=s_b$, $s_l=s_a$ be the simple relative reflections and put
\[
 \begin{aligned}
 N_s&=y_b(1)y_{-b}(-1)y_b(1),& H_s&=N_s^2,\\
 N_l&=y_a(1)y_{-a}(-1)y_a(1),& H_l&=N_l^2.
 \end{aligned}
\]
Over $\mathbb Z$, these signed monomial matrices generate a group
$\mathcal T_{\mathbb Z}$ of order $48$, with diagonal subgroup
$\langle H_s,H_l\rangle$ of order four and Weyl quotient $W(G_2)$
of order twelve.
The matrices remain distinct over every nonzero ring with $2$
invertible; write $\mathcal T$ for their local target specialization.
Their conjugation action is
\begin{equation}
 N_i y_\gamma(t)N_i^{-1}
 =y_{s_i\gamma}(\varepsilon_{i,\gamma}t),
 \qquad \varepsilon_{i,\gamma}\in\{1,-1\},\quad i=s,l,
 \label{eq:coherent-weyl}
\end{equation}
with positive-root data
\begin{equation}
 \begin{array}{c|rrrrrr}
 \gamma&b&a&c&d&e&f\\\hline
 (s_s\gamma,\varepsilon_{s,\gamma})
 &(-b,-1)&(e,-1)&(d,-1)&(c,1)&(a,1)&(f,1)\\
 (s_l\gamma,\varepsilon_{l,\gamma})
 &(c,1)&(-a,-1)&(b,-1)&(d,1)&(f,1)&(e,-1)
 \end{array}
 \label{eq:weyl-table}
\end{equation}
For negative inputs, negate the output root and retain the parameter
sign; this follows by transposition from $N_i^t=N_i^{-1}$.
Direct multiplication with independent coefficient slots proves the
positive identities. Injectivity of the central isogenies on absolute
unipotent subgroups gives the same identities in both pinned forms.
The chosen powers of $\rho$ in $y_d$ ensure sign-only transport.
The constant frame $\mathcal F_R=(y_\gamma(1))_{\gamma\in G_2}$
is integral and independent of $\rho$; its four simple signed elements
determine the whole family by Weyl words.

\subsection{Commutator relations}

The commutator coefficients use ring operations and $T,\#,N,\times$.
Our conventions are
\[
 \commR{g}{h}=g^{-1}h^{-1}gh,\qquad
 \commL{g}{h}=ghg^{-1}h^{-1},
\]
so $\commL{g}{h}=\commR{g^{-1}}{h^{-1}}$.
For $t,u\in R$ and $v,w\in R_0$, the nontrivial positive-root
relations are
\begin{align}
 \commR{y_a(v)}{y_b(t)}
 &=y_c(vt)y_d(-vt^\#)y_e(vN(t))y_f(v^2N(t)),
 \label{eq:ab}\\
 \commR{y_a(v)}{y_e(w)}
 &=y_f(vw),
 \label{eq:ae}\\
 \commR{y_c(t)}{y_b(u)}
 &=y_d(-(t\times u))y_e(T(tu^\#))y_f(T(ut^\#)),
 \label{eq:cb}\\
 \commR{y_d(t)}{y_b(u)}
 &=y_e(-T(tu)),
 \label{eq:db}\\
 \commR{y_c(t)}{y_d(u)}
 &=y_f(T(tu)).
 \label{eq:cd}
\end{align}
The other ten unordered positive pairs commute.
Every pair of distinct nonopposite roots lies in a positive system,
so \eqref{eq:coherent-weyl} gives all sixty unordered signed pairs.
Reversing the inputs inverts the entire output word, including its
factor order.

These identities follow by multiplying the absolute matrices over
$\mathbb Z[t_0,t_1,t_2,u_0,u_1,u_2,v,w]$ and specializing the triples
to $(t,\rho(t),\rho^2(t))$ and $(u,\rho(u),\rho^2(u))$.
They hold in both pinned forms because each identity lies in one
absolute unipotent subgroup, on which the central isogenies are
injective.
Within a positive system, every correction has greater height than
either input. Collection therefore gives one factor per root in any
prescribed order, uniquely by absolute coordinates.
Section~\ref{sec:centrality} treats the presentation obtained by
amalgamating these unipotent groups.

\subsection{Local structure}
\label{subsec:local-structure}

We now prove the lifting and factorization properties.
Throughout this subsection, $(R,\rho)$ is local.
Let $T_\Lambda$ be the pinned split maximal torus of $\mathbf G_\Lambda$,
let $U^\pm_\Lambda$ be the relative unipotent subgroups, and put
$H_\Lambda=T_\Lambda(R)^\delta$ for $\Lambda=Q,P$.

\begin{lemma}\label{lem:local-cell}
The ring $R_0$ is local with maximal ideal $J_0$, and
$R_0/J_0\simeq k^{\bar\rho}$.
Reduction is surjective on root parameter groups and on
$H_\Lambda\to T_\Lambda(k)^{\tau\bar\rho}$.
If $z\in G_\Lambda(R)$ has residue in the absolute open cell, then
\[
 z=u_+hu_-,\qquad
 u_\pm\in U^\pm_\Lambda,\quad h\in H_\Lambda,
\]
uniquely, and reduction preserves these three factors.
\end{lemma}

\begin{proof}
A fixed unit has a fixed inverse, so the nonunits of $R_0$ are
exactly $J_0$.
For $\bar v\in k^{\bar\rho}$, averaging any lift $t$ gives
$T(t)/3\in R_0$ with residue $\bar v$, and this lift is a unit
when $\bar v\ne0$. Short parameters lift arbitrarily.

The character lattice of $T_\Lambda$ has a basis permuted by $\tau$:
the simple roots for $Q$ and the fundamental weights for $P$.
Lift one unit coordinate on the outer orbit and take its
$\rho$-conjugates; average a lift of the fixed coordinate.
This gives the required fixed torus lift.

The absolute multiplication map
\[
 U^+_{\rm abs}\times T_\Lambda\times U^-_{\rm abs}
 \longrightarrow\mathbf G_\Lambda
\]
is an open immersion \cite[Theorem~5.1.13]{Conrad}.
Its inverse image under $z$ is open in $\Spec R$ and contains the
closed point, hence is all of $\Spec R$.
The unique absolute factors are $\delta$-fixed because $\delta$
preserves the cell.
Ordering absolute roots by their commuting triality orbits identifies
the fixed unipotent coordinates with the relative root maps above.
Thus the factors belong to $U^\pm_\Lambda$ and $H_\Lambda$;
uniqueness also gives compatibility with reduction.
\end{proof}

\begin{proposition}[Factorization]\label{prop:factorization}
For $\Lambda=Q,P$, one has
\[
 G_\Lambda(R)=E_\Lambda(R)H_\Lambda,\qquad
 E_\Lambda(R)\trianglelefteq G_\Lambda(R),\qquad
 G_P(R)=E_P(R).
\]
Reduction $G_\Lambda(R)\to G_\Lambda(k)$ is surjective.
\end{proposition}

\begin{proof}
For a fixed torus point, the three multipliers on a short orbit
form a $\rho$-orbit, and the multiplier on a long root is fixed.
Thus $H_\Lambda$ normalizes every relative root subgroup, over both
$R$ and $k$.

The absolute Bruhat decomposition over an algebraic closure of $k$
has cells and root-coordinate maps defined over $k$, so its unique
factorizations hold over $k$ \cite[Theorem~1.4.12]{Conrad}.
The action of $\tau\bar\rho$ sends the cell at $w$ to that at
$\tau(w)$; a fixed point therefore has $w\in W(D_4)^\tau$.
This Weyl group is generated by $s_{\alpha_2}$ and
$s_{\alpha_1}s_{\alpha_3}s_{\alpha_4}$
\cite[Theorem~32]{Steinberg}, so $w$ has a fixed representative
$n_w$ given by a relative Tits word.
For $\bar z\in G_\Lambda(k)$, write its unique factorization as
\[
 \bar z=\bar u_w n_w\bar h\bar u,\qquad
 \bar u_w\in U_w(k),\quad
 \bar h\in T_\Lambda(k),\quad
 \bar u\in U^+_{\rm abs}(k),
\]
where $U_w=U^+_{\rm abs}\cap n_wU^-_{\rm abs}n_w^{-1}$.
These unipotent subgroups are $\tau\bar\rho$-stable, so uniqueness
makes all three factors fixed.
Their orbit coordinates give
$G_\Lambda(k)=E_\Lambda(k)T_\Lambda(k)^{\tau\bar\rho}$.

Lemma~\ref{lem:local-cell} lifts the elementary and torus factors.
For $z\in G_\Lambda(R)$, choose $e_0h_0$ with the same residue.
The remainder $(e_0h_0)^{-1}z$ reduces to the identity and hence lies
in $U^+_\Lambda H_\Lambda U^-_\Lambda$ by the same lemma.
Since $H_\Lambda$ normalizes $E_\Lambda(R)$, this proves
$G_\Lambda(R)=E_\Lambda(R)H_\Lambda$ and normality.
Lifting any residue factorization also proves surjectivity of reduction.

For $P$, the simple coroots form a cocharacter basis, so every fixed
torus point is $h_{b,P}(r)h_{a,P}(s)$ with
$r\in R^\times$, $s\in R_0^\times$, where
\[
 \begin{aligned}
 h_{b,P}(r)
 &=h_{\alpha_4,P}(r)
   h_{\alpha_1,P}(\rho(r))h_{\alpha_3,P}(\rho^2(r)),\\
 h_{a,P}(s)&=h_{\alpha_2,P}(s).
 \end{aligned}
\]
Both factors are elementary: for $v\in R_\gamma^\times$, put
\[
 \begin{aligned}
 w_{\gamma,\Lambda}(v)
 &=y_{\gamma,\Lambda}(v)y_{-\gamma,\Lambda}(-v^{-1})
   y_{\gamma,\Lambda}(v),\\
 h_{\gamma,\Lambda}(v)
 &=w_{\gamma,\Lambda}(v)w_{\gamma,\Lambda}(1)^{-1}.
 \end{aligned}
\]
For $a$ this is the absolute rank-one identity; for $b$ it is the
product of three commuting rank-one identities.
Hence $H_P\subseteq E_P(R)$ and $G_P(R)=E_P(R)$.
\end{proof}

%-----------------------------
%-----------------------------
%-----------------------------

\medskip
\section{Normalization of the constant root elements}
\label{sec:normalization}
\label{sec:residual}\label{sec:d4-tits-normalization-v4}

We normalize the constant frame by linear changes of basis, then
correct the resulting matrix to a proper similitude with triality-fixed
adjoint class. Throughout this section, $(S,\rho_S)$ is arbitrary,
$(A,\rho_A)$ is local, and $\Psi:E_S\to E_A$ is a homomorphism.
All divisions take place in $A$.

\begin{definition}[Residual frame]\label{def:prepared}
The map $\Psi$ has a \emph{liftable prepared residual frame} if some
$h\in G_A$ satisfies
\[
 \operatorname{red}_A
 \bigl(h^{-1}\Psi(y_{\gamma,S}(1))h\bigr)=y_{\gamma,k}(1)
 \qquad(\gamma\in G_2),\qquad k=A/\Jac(A).
\]
If a residual conjugator $\bar h$ is specified, $h$ must lift it.
It suffices to check $\gamma\in\{\pm a,\pm b\}$, since these
constants determine the others by Weyl words.
\end{definition}

A sufficient residual condition also records a coefficient map.
Let $\nu:S\to k$ be unital with
$\nu\rho_S=\bar\rho_A^\epsilon\nu$, $\epsilon\in\{1,-1\}$.
Then $I=\ker\nu$ is a $\rho_S$-stable prime, $B=S/I$ inherits
$\rho_B$, and $\bar\nu:B\hookrightarrow k$ has the same compatibility.
Let $r:E_S\twoheadrightarrow H=\operatorname{im}(E_S\to E_B)$
be the coefficient quotient map with codomain restricted to its image.
The condition is
\begin{equation}
 \operatorname{red}_A\circ\Psi
 =\Int(\bar h)\circ(\Phi_{\bar\nu}^\epsilon|_H)\circ r,
 \qquad \bar h\in\operatorname{im}(G_A\to G_k).
 \label{eq:typed-rs}
\end{equation}
Here $H$ cannot in general be replaced by $E_B$: fixed parameters
need not lift through a source quotient without $1/3$.

The maps $\Phi_f^\epsilon$ have a pinned construction independent
of the coefficient converse. For an equivariant map $f:S\to A$,
restrict $\mathbf G_Q(f)$ to fixed points.
For $f\rho_S=\rho_A^{-1}f$,
restrict $d_0\circ\mathbf G_Q(f)$, where the pinned automorphism
$d_0$ exchanges $\alpha_1,\alpha_3$ and fixes $\alpha_2,\alpha_4$.
Since $d_0\tau d_0^{-1}=\tau^{-1}$,
\[
 (d_0\circ\mathbf G_Q(f))\circ(\tau\rho_S)
 =(\tau\rho_A)^{-1}\circ(d_0\circ\mathbf G_Q(f)),
\]
so fixed points map to fixed points. On $y_{b,S}(t)$, the transposition
exchanges the two slots $\rho_A^{-1}(f(t))$ and
$\rho_A^{-2}(f(t))$, giving $y_{b,A}(f(t))$.
The same root rule holds on the other simple signed subgroups and
hence, by their Tits words and \eqref{eq:coherent-weyl}, on all roots.
The construction works in both pinned forms. Applied to $\bar\nu$,
it shows that \eqref{eq:typed-rs}
implies Definition~\ref{def:prepared}.

For normalization, assume only Definition~\ref{def:prepared}.
Choose $h$ and put $\varphi=\Int(h^{-1})\circ\Psi:E_S\to E_A$;
Proposition~\ref{prop:factorization} ensures that the image remains
in $E_A$. Write $R=A$, $\rho=\rho_A$, $J=\Jac(R)$, and
$q:\GL_8(R)\to\operatorname{PGL}_8(R)$. Put
\[
 \begin{aligned}
 u_s&=y_{b,S}(1),&
 U_s&=X_{\alpha_4}(1)X_{\alpha_1}(1)X_{\alpha_3}(1),\\
 u_l&=y_{a,S}(1),& U_l&=X_{\alpha_2}(1).
 \end{aligned}
\]
The integral Tits group specializes to
$\iota_S:\mathcal T_{\mathbb Z}\to E_S$, not necessarily injectively.
We identify $\mathcal T_{\mathbb Z}$ with its target matrix group
$\mathcal T$ when indexing these source words.

\subsection{Representatives and averaging}
\label{sec:tits-averaging}

Compatible matrix representatives allow averaging over the Tits group.

\begin{lemma}[Matrix representatives]
\label{lem:tits-coherence}\label{lem:tits-pointwise-lifts}
For each $t\in\mathcal T$, the element $\varphi(\iota_S(t))$ has
a unique representative $A_t\in\SL_8(R)$ with $A_t\equiv t\pmod J$.
These representatives are orthogonal and satisfy $A_tA_u=A_{tu}$.
Likewise, $\varphi(u_s),\varphi(u_l)$ have unique determinant-one
representatives $B_s\equiv U_s$, $B_l\equiv U_l\pmod J$, which are
orthogonal. The integral root-word relations used below hold as exact
matrix identities for these representatives.
\end{lemma}

\begin{proof}
Every elementary target element has an orthogonal determinant-one
representative. For a prepared constant, its reduction is $\mu U$,
where $U$ is the standard matrix and $\mu^2=1$ in $k$.
Multiplication by $-I_8$, if necessary, gives the required reduction
without changing the determinant.

Two determinant-one representatives with this reduction differ by
$\lambda\in1+J$ with $\lambda^8=1$. Since
\[
 0=(\lambda-1)(1+\lambda+\cdots+\lambda^7),\qquad
 1+\lambda+\cdots+\lambda^7\equiv8\pmod J,
\]
one has $\lambda=1$.
The two sides of each retained source identity have the same
projective class, determinant one, and equal standard reductions.
The same argument gives exact equality, including $A_tA_u=A_{tu}$.
\end{proof}

\begin{proposition}[Averaging]\label{prop:full-tits-averaging}
The matrix
\[
 C_T=\frac1{48}\sum_{t\in\mathcal T}A_tt^{-1}
\]
belongs to $\GL_8(R)$, satisfies $C_T\equiv I_8\pmod J$, and
$C_T^{-1}A_sC_T=s$ for every $s\in\mathcal T$.
\end{proposition}

\begin{proof}
Since $48\in R^\times$ and each summand is congruent to $I_8$,
$C_T\equiv I_8\pmod J$ and $C_T$ is invertible.
Reindexing by $u=st$ gives
\[
 A_sC_T=\frac1{48}\sum_{t\in\mathcal T}A_{st}t^{-1}
       =\frac1{48}\sum_{u\in\mathcal T}A_uu^{-1}s=C_Ts.
\]
\end{proof}

Set $\varphi_T=\Int(q(C_T)^{-1})\circ\varphi:E_S\to\operatorname{PGL}_8(R)$
and $X_\gamma=C_T^{-1}B_\gamma C_T$ for $\gamma=s,l$.
The Tits representatives are standard, and $X_s,X_l$ have determinant
one and standard reductions. They satisfy the identities of
Appendix~\ref{app:certificate-words}, where auxiliary roots are
specified Weyl conjugates and integer parameters are powers.
This intermediate change of basis need not preserve orthogonality.

\subsection{The Tits centralizer}
\label{sec:tits-centralizer}

Further changes of basis must centralize the normalized Tits group.

\begin{proposition}\label{prop:tits-centralizer}
Let $R$ be a commutative ring with $2\in R^\times$.
Group the basis into $V_i=\langle e_i,e_{-i}\rangle$, $1\leq i\leq4$,
and put $P_2=\left(\begin{smallmatrix}0&1\\1&0\end{smallmatrix}\right)$.
Relative to $R^8=V_1\oplus\cdots\oplus V_4$, the centralizer
of $\mathcal T$ in $M_8(R)$ consists of
\[
 A(a,b,c,d)=\operatorname{diag}
 \bigl(bI_2+aP_2,bI_2+aP_2,bI_2+aP_2,dI_2+cP_2\bigr),
 \qquad a,b,c,d\in R.
\]
Its algebra is $R[C_2]\times R[C_2]\simeq R^4$, and
$\det A(a,b,c,d)=(b^2-a^2)^3(d^2-c^2)$.
\end{proposition}

\begin{proof}
On the four planes,
\[
 H_s=\operatorname{diag}(-I_2,-I_2,I_2,I_2),\qquad
 H_l=\operatorname{diag}(I_2,-I_2,-I_2,I_2).
\]
Their distinct pairs of signs force a commuting matrix to be
$\operatorname{diag}(D_1,D_2,D_3,D_4)$. The remaining equations are
\[
 D_1=D_2=D_3,\qquad D_1P_2=P_2D_1,\qquad D_4P_2=P_2D_4.
\]
A matrix commuting with $P_2$ has equal diagonal and equal off-diagonal
entries, giving the stated form. The determinant follows blockwise,
and $A(a,b,c,d)\mapsto(b+a,b-a,d+c,d-c)$ gives the algebra isomorphism.
\end{proof}

In the local target, $A(a,b,c,d)\equiv I_8\pmod J$ is equivalent to
$a,c\in J$, $b,d\in1+J$, which also ensures invertibility.

\subsection{Normalization of the simple roots}
\label{sec:d4-root-images-v3}\label{sec:d4-root-images-v2}
\label{subsec:d4-post-averaging-admissibility-v3}
\label{subsec:d4-simple-root-calculation-v3}

We normalize the long root first, then the short root within the
remaining centralizer.
A pair $(X_s,X_l)\in\SL_8(R)^2$ is \emph{admissible} if it reduces
to $(U_s,U_l)$ and satisfies the stabilizer, inversion, cubic,
commuting, and mixed relations of Appendix~\ref{app:certificate-words},
with its specified Weyl words. The preceding construction gives
such a pair. Commutators here are left commutators.

Put $W_s=N_lN_sN_lN_sN_l$ and $W_l=N_sN_lN_sN_lN_s$.
The stabilizer equations $X_sW_s=W_sX_s$, $X_lW_l=W_lX_l$
have the $24$- and $16$-coordinate matrix forms in
Appendix~\ref{app:certificate-forms}, with coordinates
$s_1,\ldots,s_{24}$ and $l_1,\ldots,l_{16}$.
The remaining equations and their selected components are specified
in Appendix~\ref{app:certificate}.

\smallskip\noindent\emph{The long root.}
\label{subsec:d4-long-centralizer-gauge}
Proposition~\ref{prop:d4-long-root-form} identifies the long matrix
as $X_l(p,z)$ with $p\in1+J$, $z\in J$, and $p^2-p-z^2=0$.
The matrix $C_l=A(-z,p,0,1)$ centralizes $\mathcal T$,
reduces to the identity, and satisfies
\[
 X_l(p,z)C_l=C_lU_l.
\]
It is invertible because $p^2-z^2=p\in R^\times$.
Conjugate both matrices by $C_l^{-1}$ and retain the notation
$X_s,X_l$; now $X_l=U_l$.
Since $A(a,b,c,d)$ commutes with $U_l$ exactly when $a=0$,
the remaining corrections have the form $A(0,b,c,d)$.

\smallskip\noindent\emph{The short root.}
\label{subsec:d4-short-root-normalization}
An explicit centralizer orbit provides a comparison solution;
the following lemma gives local uniqueness.

\begin{lemma}[Local uniqueness]
\label{lem:d4-divided-difference-uniqueness}
Let $F\in R[y_1,\ldots,y_n,t]^n$, $y_0\in R^n$, $t_0\in R$,
and suppose $\det(\partial F/\partial y)(y_0,t_0)\in R^\times$.
For fixed $t\equiv t_0\pmod J$, there is at most one $y\in R^n$
with $y-y_0\in JR^n$ and $F(y,t)=0$.
\end{lemma}

\begin{proof}
Telescoping monomials gives a polynomial matrix $M$ with
\[
 F(y,t)-F(y',t)=M(y,y',t)(y-y'),\qquad
 M(y,y,t)=\frac{\partial F}{\partial y}(y,t).
\]
For two such solutions, $M(y,y',t)$ reduces to the invertible
Jacobian at $(y_0,t_0)$, so $M$ is invertible and $y=y'$.
\end{proof}

With $X_l=U_l$, take the $23$ polynomials $F_1,\ldots,F_{23}$
from Appendix~\ref{app:certificate-short}, hold $s_2$ fixed,
and use the other coordinates as $y$.
At the standard tuple $\mathbf s^0$, Appendix~\ref{app:certificate-minor}
gives
\[
 \left.
 \det\frac{\partial(F_1,\ldots,F_{23})}
          {\partial(s_1,s_3,\ldots,s_{24})}
 \right|_{\mathbf s^0}=384=2^7\cdot3\in R^\times.
\]
Thus a near-standard solution is unique for each fixed $s_2$.

Set $\kappa=s_2\in1+J$ and $C_\kappa=A(0,1,0,\kappa)$.
This matrix centralizes $\mathcal T$ and $U_l$ and reduces to $I_8$.
Hence $X_s(\kappa)=C_\kappa^{-1}U_sC_\kappa$ satisfies all
fixed-long relations and reduces to $U_s$.
Its $(3,4)$ entry, the coordinate $s_2$, equals $\kappa$.
Local uniqueness gives $X_s=X_s(\kappa)$, so conjugation by
$C_\kappa$ sends $X_s$ to $U_s$ while preserving $U_l$ and
$\mathcal T$.
The Jacobian is used only for uniqueness; the explicit comparison
solution requires no completeness or Henselian hypothesis.

\begin{theorem}[Simple-root normalization]
\label{thm:d4-simple-root-normalization-v3}
Every admissible pair is simultaneously conjugate to $(U_s,U_l)$
by $C_{\rm root}\in\GL_8(R)$ congruent to $I_8$ modulo $J$
and centralizing $\mathcal T$:
\[
 C_{\rm root}^{-1}X_sC_{\rm root}=U_s,\qquad
 C_{\rm root}^{-1}X_lC_{\rm root}=U_l.
\]
\end{theorem}

\begin{proof}
For the original pair, choose $C_l$ as above and let $\kappa$ be
the $(3,4)$ entry of $C_l^{-1}X_sC_l$.
Then $C_{\rm root}=C_lC_\kappa^{-1}$ has the required properties.
\end{proof}

Put $C_{\rm norm}=C_TC_{\rm root}$.
The map $\varphi_{\rm norm}=\Int(q(C_{\rm norm})^{-1})\circ\varphi$
into $\operatorname{PGL}_8(R)$ fixes the Tits words and simple
parameter-one roots, hence the complete frame.
For the form correction, set
\begin{equation}
 \begin{gathered}
 C=C_{\rm norm}^{-1},\qquad \mathcal S=\{N_s,N_l,U_s,U_l\},\\
 \varphi(g_S)=[C^{-1}gC]\qquad(g\in\mathcal S),
 \end{gathered}
 \label{eq:d4-final-orientation}
\end{equation}
where $[M]=q(M)$ and $g_S$ is the corresponding integral source word.
The matrix $C$ is still only a linear change of basis.

\subsection{The conjugating element}
\label{sec:d4-final-geometric-correction}

We replace $C$ by $C'=DC$, with $D$ centralizing $\mathcal S$,
to obtain a proper similitude without changing its transport of
the frame. For matrices, put $M^\#=J_8M^tJ_8$;
then $(MN)^\#=N^\#M^\#$.

\smallskip\noindent\emph{The frame centralizer.}
Let $Q_0$ fix $e_i$ for $i=\pm1,\pm2,\pm3$ and exchange
$e_4,e_{-4}$. It is an orthogonal involution of determinant $-1$.
The complementary self-adjoint projectors
\[
 P_+=\frac{I_8+Q_0}{2},\qquad P_-=\frac{I_8-Q_0}{2}
\]
have ranks seven and one: $P_-R^8=R(e_4-e_{-4})$, while
$P_+R^8$ has basis $e_1,e_2,e_3,e_{-1},e_{-2},e_{-3},e_4+e_{-4}$.

\begin{proposition}\label{prop:d4-final-common-centralizer}
Over every commutative ring $R$ with $2\in R^\times$,
\begin{equation}
 \Cent_{M_8(R)}(\mathcal S)=RP_+\oplus RP_-=RI_8+RQ_0,
 \qquad \det(uP_++vP_-)=u^7v.
 \label{eq:d4-final-centralizer}
\end{equation}
\end{proposition}

\begin{proof}
In the Tits centralizer $A(a,b,c,d)$,
commutation with $U_l$ forces $a=0$.
For $A(0,b,c,d)$, the entries of $AU_s-U_sA$ are $b-c-d$ at
$(3,4),(3,-4),(4,-3),(-4,-3)$ and zero elsewhere.
Thus $b=c+d$; setting $b=u$, $c=(u-v)/2$, $d=(u+v)/2$ gives
$A=uP_++vP_-$. The determinant follows on the two summands.
\end{proof}

\begin{lemma}[Frame centralizer]
\label{lem:d4-final-proper-centralizer}\label{lem:ambient-center}
For every nonzero commutative ring $B$ with $2\in B^\times$,
\[
 \Cent_{\operatorname{GO}_8^+(B)}(\mathcal S_B)=B^\times I_8,
 \qquad
 \Cent_{\mathbf G_Q(B)}([\mathcal S_B])
 =\Cent_{\mathbf G_Q(B)}(\mathcal F_B)=1.
\]
Here $\mathcal F_B$ is the integral constant frame, independent
of a coefficient automorphism.
For a coefficient pair $(B,\rho_B)$, this gives
$\Cent_{\mathbf G_Q(B)}(E_B)=1$ and $Z(E_B)=Z(G_B)=1$.
\end{lemma}

\begin{proof}
A proper similitude in the matrix centralizer has the form
$Z=pP_++rP_-$ with $p,r\in B^\times$.
Since $Z^\#=Z$, its multiplier $\mu$ satisfies $p^2=r^2=\mu$.
Properness gives $p^7r=\mu^4=p^8$, hence $r=p$ and $Z$ is scalar.

Now lift an adjoint point centralizing $[\mathcal S_B]$ to a proper
similitude $Z$ after a faithfully flat base change.
The negative simple-root matrices $U_{-s}=U_s^t$, $U_{-l}=U_l^t$
satisfy
\[
 U_{-s}=U_sN_s^{-1}U_s,\qquad U_{-l}=U_lN_l^{-1}U_l.
\]
Thus $Z$ commutes projectively with all four simple signed root
matrices. For each of them, $ZUZ^{-1}=\xi_UU$; taking traces gives
$8=8\xi_U$, hence $\xi_U=1$.
Their Tits words give literal commutation with $\mathcal S_B$,
so $Z$ is scalar by the first part.
Triviality of the adjoint point descends faithfully flatly.
An element centralizing the frame also centralizes its simple
signed elements and their Tits words.
The remaining assertions follow from $\mathcal F_B\subset E_B\subset G_B$.
\end{proof}

\smallskip\noindent\emph{Correction of the form.}
\begin{theorem}\label{thm:d4-final-geometric-correction}
Let $C\in\GL_8(R)$ satisfy $C\equiv I_8\pmod J$ and
$[C^{-1}gC]\in G_R$ for $g\in\mathcal S$.
Then $CC^\#=uP_++vP_-$ for unique $u,v\in1+J$.
Set
\[
 s=\frac{\det C}{u^3},\qquad
 D=P_++\frac{u}{s}P_-,\qquad C'=DC.
\]
The matrix $C'$ is a proper similitude with multiplier $u$,
$C'\equiv I_8\pmod J$, $[C']\in G_R$, and
\begin{equation}
 C'^{-1}gC'=C^{-1}gC\qquad(g\in\mathcal S).
 \label{eq:d4-final-unchanged}
\end{equation}
Its class is the unique element of $\mathbf G_Q(R)$ with these
projective transport identities.
\end{theorem}

\begin{proof}
Set $Y_g=C^{-1}gC$.
Its adjoint class has a proper-similitude representative because
$R$ is local. The matrix $Y_g$ differs from it by a unit scalar,
so $Y_g^\#Y_g=m_gI_8$ and $\det Y_g=m_g^4$.
Since $\det Y_g=1$ and $m_g\equiv1\pmod J$, the factorization
\[
 0=m_g^4-1=(m_g-1)(1+m_g+m_g^2+m_g^3)
\]
has a unit second factor. Hence $m_g=1$ and $Y_g$ is orthogonal.

Put $N=CC^\#$.
Expanding $Y_g^\#Y_g=I_8$ gives
$C^\#g^{-1}N^{-1}gC=I_8$, hence $g^{-1}N^{-1}g=N^{-1}$.
The centralizer calculation and $N\equiv I_8\pmod J$ give
\begin{equation}
 CC^\#=uP_++vP_-,\qquad u,v\in1+J,\qquad
 (\det C)^2=u^7v.
 \label{eq:d4-final-defect}
\end{equation}
For $s=\det C/u^3$ and $r=u/s$,
\begin{equation}
 s^2=uv,\qquad r^2=u/v.
 \label{eq:d4-final-square-identities}
\end{equation}
Thus the required ratio is obtained from $\det C$ without choosing
a square root.
The self-adjoint matrix $D=P_++rP_-$ centralizes $\mathcal S$
and reduces to $I_8$. Consequently,
\[
 C'C'^\#=D(CC^\#)D=uP_++r^2vP_-=uI_8,\qquad
 \det C'=r\det C=u^4.
\]
Invertibility also gives $C'^\#C'=uI_8$, so $C'$ is a proper
similitude with multiplier $u$.
Since $D$ centralizes $\mathcal S$, \eqref{eq:d4-final-unchanged} holds.

Let $c=[C']\in\mathbf G_Q(R)$.
Both $[g]$ and $c^{-1}[g]c$ are fixed by $\delta=\tau\rho$.
Applying $\delta$ shows that $\delta(c)c^{-1}$ centralizes
$[\mathcal S]$, so Lemma~\ref{lem:ambient-center} gives $c\in G_R$.
If $c_1,c_2$ have the same projective transport identities, then
$c_2c_1^{-1}$ centralizes $[\mathcal S]$; the same lemma proves
uniqueness.
\end{proof}

Apply the theorem to \eqref{eq:d4-final-orientation} and put
$\psi=\Int(c)\circ\varphi:E_S\to E_R$.
This map fixes the complete constant frame and, as a map into
$\operatorname{PGL}_8(R)$, equals
$\Int([D])\circ\varphi_{\rm norm}$.
Returning to $\varphi=\Int(h^{-1})\circ\Psi$ gives
\[
 g_0=hc^{-1}\in G_R,\qquad
 \Psi=\Int(g_0)\circ\psi,\qquad \overline{g_0}=\bar h.
\]
Any two such elements transport the same indexed frame to
$\Psi(\mathcal F_S)$, so their quotient centralizes $\mathcal F_R$.
Lemma~\ref{lem:ambient-center} proves uniqueness of $g_0$.

%-----------------------
%-----------------------
%-----------------------

\medskip

\section{Recovery of the coefficient homomorphism}
\label{sec:recognition}\label{sec:d4-root-recognition-v6}

We recognize the images of the root subgroups and recover their common
coefficient map. Let $\psi:E_S\to E_A$ be a homomorphism fixing the
constant frame, where $(S,\rho_S)$ is arbitrary and $2\in A^\times$.
Put $S_0=S^{\rho_S}$ and $A_0=A^{\rho_A}$, and retain the parameter
groups $S_\gamma,A_\gamma$ from Section~\ref{sec:coordinates}.
No locality assumption on $A$ is needed.

\subsection{Root recognition}

Write $U_\gamma$ for the standard matrix representing
$y_{\gamma,A}(1)$ and $q_A:\GL_8(A)\to\operatorname{PGL}_8(A)$
for the projective quotient.

\begin{lemma}[Recognition]\label{lem:recognition}
For every signed relative root $\gamma$,
\[
 \psi\bigl(y_{\gamma,S}(S_\gamma)\bigr)
 \subseteq y_{\gamma,A}(A_\gamma).
\]
The resulting parameter is unique and belongs to $A$ for a short
root and to $A_0$ for a long root.
\end{lemma}

\begin{proof}
Every image has an orthogonal determinant-one representative $Y$,
obtained from an elementary root word. If its projective class
commutes with a constant root element, then
$YU_\beta Y^{-1}=\zeta U_\beta$ for some $\zeta\in A^\times$.
Since $\operatorname{tr}U_\beta=8\in A^\times$, taking traces gives
$\zeta=1$. Thus commutation with the fixed frame gives linear
equations on $Y$.

\smallskip\noindent\emph{The long simple root.}
For $v\in S_0$, the element $y_{a,S}(v)$ commutes with the constants
at $-b,a,c,d,-d,-e,f$. With
$M_a=M_{\alpha_2}=E_{2,3}-E_{-3,-2}$ and $Q_0$ from
Section~\ref{sec:d4-final-geometric-correction}, the common matrix
centralizer is
\[
 Y=-z_0M_a+z_1Q_0+z_2I_8,\qquad z_0,z_1,z_2\in A.
\]
The elimination has $61$ pivots equal to $\pm1$ and no residual
equations, so this description is valid after arbitrary base change.
Orthogonality and determinant one give
\[
 2z_1z_2=0,\qquad z_1^2+z_2^2=1,\qquad
 (z_1+z_2)^7(z_2-z_1)=1.
\]
For $p=z_1+z_2$ and $q=z_2-z_1$, these imply
$p^2=q^2=1$ and $p^7q=pq=1$. Hence $p=q$, so
$z_1=0$ and $z_2^2=1$. Rescaling gives
\[
 z_2^{-1}Y=X_{\alpha_2}(w),\qquad w=-z_0/z_2.
\]
Its adjoint class is fixed by $\delta_A=\tau\rho_A$, and $\tau$
fixes the pinned root map $x_{\alpha_2,Q}$. Thus
$x_{\alpha_2,Q}(\rho_A(w))=x_{\alpha_2,Q}(w)$.
Injectivity gives $w\in A_0$ and
$\psi(y_{a,S}(v))=y_{a,A}(w)$.

\smallskip\noindent\emph{The short simple root.}
For $t\in S$, the element $y_{b,S}(t)$ commutes with the constants
at $b,-a,e,f,-f$. Each outer absolute factor satisfies these
commutations separately. Their common matrix centralizer is
{\small
\[
 Y=
 \begin{pmatrix}
 z_4+z_6&-z_3&0&0&0&0&0&0\\
 0&z_4+z_6&0&0&0&0&0&0\\
 0&0&z_4+z_6&-z_1-z_2-z_5&0&0&z_0&z_1\\
 0&0&0&z_6&0&0&z_2&z_4\\
 0&0&0&0&z_4+z_6&0&0&0\\
 0&0&0&0&z_3&z_4+z_6&0&0\\
 0&0&0&0&0&0&z_4+z_6&0\\
 0&0&0&z_4&0&0&z_5&z_6
 \end{pmatrix}.
\]
}
Here $z_0,\ldots,z_6\in A$ are unrestricted at the linear stage:
elimination has $57$ pivots equal to $\pm1$ and no residual equations.

The $(e_4,e_4)$- and $(e_4,e_{-4})$-entries of $Y^tJ_8Y=J_8$,
together with determinant one, give
\[
 2z_4z_6=0,\qquad z_4^2+z_6^2=1,\qquad
 (z_4+z_6)^7(z_6-z_4)=1.
\]
As in the long-root case, $z_4=0$ and $z_6^2=1$.
Rescale by $z_6^{-1}$ and rename the parameters, so that $z_6=1$.
The remaining orthogonality conditions are
\[
 z_1+z_2=0,\qquad z_0+z_2z_5=0.
\]
They come from the $(e_4,e_{-3})$- and $(e_{-3},e_{-3})$-entries,
with $2$ cancelled in the latter; all other entries follow.
Hence
\[
 Y=X_{\alpha_4}(z_1)X_{\alpha_1}(-z_3)X_{\alpha_3}(-z_5).
\]
The pinned automorphism $\tau$ cycles these three simple root maps
with positive signs. Since $q_A(Y)$ is $\delta_A$-fixed,
uniqueness of absolute unipotent coordinates gives
\[
 -z_3=\rho_A(z_1),\qquad -z_5=\rho_A^2(z_1).
\]
Both products have diagonal entries one, so their projective equality
has no scalar ambiguity. Thus
$\psi(y_{b,S}(t))=y_{b,A}(z_1)$.

The fixed simple signed constants determine the Tits words.
Equation~\eqref{eq:coherent-weyl} therefore transports these two
recognition statements to every signed root.
Uniqueness follows from injectivity of the root maps.
\end{proof}

\subsection{The ring map}

\begin{theorem}[Coefficient recovery]\label{thm:recovery}
There is a unique unital ring homomorphism $\eta:S\to A$ such that
\[
 \psi(y_{\gamma,S}(t))=y_{\gamma,A}(\eta(t))
 \qquad(\gamma\in G_2,\ t\in S_\gamma).
\]
It satisfies $\eta(S_0)\subseteq A_0$ and
$\eta\circ T_S=T_A\circ\eta$.
If the coefficient pairs coincide and $\psi$ is an automorphism,
then $\eta$ is a ring automorphism.
\end{theorem}

\begin{proof}
Recognition defines parameter maps
$\eta_\gamma:S_\gamma\to A_\gamma$ by
$\psi(y_{\gamma,S}(t))=y_{\gamma,A}(\eta_\gamma(t))$.
Root additivity makes them additive, and the fixed frame gives
$\eta_\gamma(1)=1$. Applying $\psi$ to
\eqref{eq:coherent-weyl} yields
\[
 \eta_{s_i\gamma}(\varepsilon_{i,\gamma}t)
 =\varepsilon_{i,\gamma}\eta_\gamma(t).
\]
Additivity removes the signs. Weyl transitivity gives one short
map $\eta:S\to A$ and one long map $\eta_0:S_0\to A_0$,
both additive and taking $1$ to $1$.

Apply $\psi$ to \eqref{eq:ab} and compare the $c$-factor in
the unique ordered positive-root coordinates:
\[
 \eta(vt)=\eta_0(v)\eta(t)\qquad(v\in S_0,\ t\in S).
\]
Taking $t=1$ gives $\eta_0=\eta|_{S_0}$.
Thus $\eta(S_0)\subseteq A_0$, and multiplication is preserved
when one factor belongs to $S_0$.

Comparison of the $d$-factor in \eqref{eq:cb} gives
\[
 \eta(t\times_S u)=\eta(t)\times_A\eta(u).
\]
Since $t\times1=T(t)-t$, setting $u=1$ proves trace preservation.

Multiplication of arbitrary elements is recovered from the
integral identity
\[
 2tu=t\times u+T(t)u+T(u)t-T(t\times u).
\]
The trace values lie in $S_0$, so the mixed products are already
preserved. Applying $\eta$ gives
\begin{align*}
 2\eta(tu)
 &=\eta(t)\times_A\eta(u)
   +T_A(\eta(t))\eta(u)+T_A(\eta(u))\eta(t)\\
 &\quad-T_A\bigl(\eta(t)\times_A\eta(u)\bigr)\\
 &=2\eta(t)\eta(u).
\end{align*}
Cancel $2$ in $A$. No division in $S$ is used.
A short-root coordinate determines $\eta$ uniquely, and the root
generators determine $\psi$.

If $\psi$ is an automorphism, its inverse fixes the same frame.
Applying recovery to $\psi^{-1}$ gives a ring map $\theta:S\to S$.
Evaluation of the two composites on a short root gives
$\theta\eta=\eta\theta=\id$, so $\eta$ is a ring automorphism.
\end{proof}

The converse construction from a coefficient map is proved in
Section~\ref{sec:converse}.

\subsection{Relative rigidity}\label{sec:relative}

\begin{theorem}[Relative rigidity]\label{thm:relative}
Let $(S,\rho_S)$ be arbitrary, let $(A,\rho_A)$ be local, and suppose
$\Psi:E_S\to E_A$ satisfies Definition~\ref{def:prepared}.
There are unique $g\in G_A$ and a unital ring homomorphism
$\eta:S\to A$ satisfying $\eta T_S=T_A\eta$,
$\eta(S_0)\subseteq A_0$, and
\begin{equation}
 \Psi(y_{\gamma,S}(t))
 =g\,y_{\gamma,A}(\eta(t))g^{-1}
 \qquad(\gamma\in G_2,\ t\in S_\gamma).
 \label{eq:relative-root-rule}
\end{equation}
For any preparation $h$, one has $\bar g=\bar h$.
Under \eqref{eq:typed-rs}, also $\overline{\eta(t)}=\nu(t)$
for every $t\in S$.
No injectivity or surjectivity of $\Psi$ is required.
\end{theorem}

\begin{proof}
Section~\ref{sec:normalization} gives $g\in G_A$ with
$\bar g=\bar h$ such that
$\psi=\Int(g^{-1})\circ\Psi$ fixes the frame.
Theorem~\ref{thm:recovery} supplies $\eta$.

For two pairs, evaluating at $1$ makes the quotient of their
conjugators centralize the frame.
Lemma~\ref{lem:ambient-center} identifies the conjugators,
and a short-root coordinate then identifies the ring maps.
Under \eqref{eq:typed-rs}, reduction and cancellation of
$\Int(\bar h)$ on a short root give
\[
 y_{\gamma,k}(\overline{\eta(t)})=y_{\gamma,k}(\nu(t)),
\]
proving the residue assertion.
\end{proof}

Uniqueness holds over every target with $2$ invertible whenever
a pair satisfying the root rule exists.
In particular, an equivariant map
$f:(A,\rho_A)\to(D,\rho_D)$ sends the data $(g,\eta)$ to
$(f_*(g),f\circ\eta)$ for the composite
$f_*\circ\Psi:E_S\to E_D$.

\subsection{Descent}\label{subsec:descent}

We use this uniqueness to descend the conjugator and coefficient
map through a faithfully flat extension of the target.

\begin{proposition}[Descent]\label{prop:descent}
Let $(S,\rho_S)$ and $(B,\rho_B)$ be coefficient pairs, with
$2\in B^\times$, and let $f:B\to C$ be faithfully flat.
Fix $\Psi:E_S\to E_B$.
Suppose $g_C\in\mathbf G_Q(C)$ transports the base-changed
constant frame to $\mathbf G_Q(f)\Psi(\mathcal F_S)$.
Then $g_C$ descends uniquely to $g\in G_B$ with the same
transport identities. No operator on $C$ is required.

If $C$ also carries an automorphism $\rho_C$ with
$\rho_C^3=\id$ and $f\rho_B=\rho_Cf$, and a unital trace-preserving
ring map $\eta_C:S\to C$ satisfies
$\eta_C(S_0)\subseteq C^{\rho_C}$ and
\[
 \mathbf G_Q(f)\Psi(y_{\gamma,S}(t))
 =g_Cy_{\gamma,C}(\eta_C(t))g_C^{-1},
 \qquad \gamma\in G_2,\quad t\in S_\gamma,
\]
then $\eta_C=f\circ\eta$ for a unique unital ring homomorphism
$\eta:S\to B$ satisfying $\eta T_S=T_B\eta$,
$\eta(S_0)\subseteq B^{\rho_B}$, and the corresponding root rule
for $\Psi$. Neither $B$ nor $C$ need be local.
\end{proposition}

\begin{proof}
Put $C_2=C\otimes_B C$, with
$p_1(c)=c\otimes1$ and $p_2(c)=1\otimes c$.
The two pullbacks of $g_C$ transport the same indexed frame,
so they agree by Lemma~\ref{lem:ambient-center}.
Affine faithfully flat descent gives $g\in\mathbf G_Q(B)$.
Applying $\tau\rho_B$ to the descended transport identities
and using the same uniqueness gives $g\in G_B$.

For the second assertion, equivariance makes
$\rho_C\otimes\rho_C$ an automorphism of $C_2$.
Pulling back the root rule and cancelling the common conjugator
gives $p_1\eta_C(t)=p_2\eta_C(t)$ by injectivity of the short-root map.
The equalizer $B\to C\rightrightarrows C_2$ yields $\eta:S\to B$.
Injectivity of $f$ detects the ring laws, trace preservation,
and the fixed-parameter condition; faithful flatness detects
the root rule. The same frame and short-coordinate arguments
prove uniqueness.
\end{proof}

Throughout, the source remains $E_S$:
descent neither extends $\Psi$ to a new source nor lifts $g$
through a nonadjoint central isogeny.

%---------------------------
%---------------------------
%---------------------------

\medskip
\section{The graded cover and coefficient actions}\label{sec:centrality}
We prove centrality of the local graded-cover kernel and use it to
construct trace-preserving coefficient substitutions on elementary
groups. The cover is defined for every coefficient pair; locality
is imposed only after its presentation. Commutators in this section
are right commutators.

\subsection{Presentation and reduction}\label{subsec:cover-presentation}
For a chamber $C$, identified with its positive root system, put
$U_C=\langle y_\gamma(R_\gamma):\gamma\in C\rangle\leq E_R$ and define
\[
 \Scover_R=\operatorname{GCov}_{G_2}(E_R),\qquad
 q_R:\Scover_R\twoheadrightarrow E_R,\qquad K_R=\ker q_R.
\]
The graded cover amalgamates the concrete $U_C$ along their common
root-generated subgroups \cite[\S7.1]{EJK}. These are unipotent
groups, without torus factors. Write $\widetilde y_\gamma$ for the
lifted root maps and $\U_C$ for the canonical image of $U_C$.

\begin{lemma}[Presentation]\label{lem:borel}
The restriction $q_R:\U_C\to U_C$ is an isomorphism for every $C$.
The group $\Scover_R$ is presented by root additivity and all
nonopposite commutator relations from Section~\ref{sec:coordinates}.
If $\widehat E_R\leq\operatorname{St}_{D_4}(R)$ is generated by the
same triality-fixed absolute orbit words, its elementary projection
induces
$\operatorname{GCov}_{G_2}(\widehat E_R)\simeq\Scover_R$.
\end{lemma}
\begin{proof}
A regular functional defining $C$ pulls back under $D_4\to G_2$ to a
regular absolute functional, since no absolute root folds to zero.
Thus all absolute factors of $U_C$ lie in one absolute unipotent
group. Its unique coordinates \cite[Theorem~5.1.13]{Conrad} give
unique relative coordinates in any order. Commutator corrections
have greater height than both inputs, so the stated relations collect
every word into one factor per root. A collected word with trivial
concrete image has all parameters zero. These relations therefore
present $U_C$, whose map into the cover is injective because its
composition with $q_R$ is the concrete inclusion. Every distinct
nonopposite pair lies in a positive system, giving the cover
presentation.

The absolute Steinberg projection is likewise injective on each
absolute unipotent subgroup, by collection \cite[\S7.2]{EJK}.
Hence $\widehat E_R\to E_R$ is an isomorphism on each relative
unipotent subgroup, and the Borel-wise comparison of
\cite[\S7.1]{EJK} identifies their graded covers, not
$\widehat E_R$ itself with either $E_R$ or $\Scover_R$.
\end{proof}
In the triality coordinates of \cite[p.~69]{EJK}, our orbit words are
\[
 y_{\pm b}(t)=z_{\pm b}(\rho(t)),\quad
 y_{\pm c}(t)=z_{\pm c}(\rho(t)),\quad
 y_{\pm d}(t)=z_{\pm d}(-t),
\]
with unchanged long charts. Negative $z$-charts are obtained by
negating the absolute roots while retaining coefficient signs and
powers of $\rho$.

Now let $(R,\rho)$ be local. Put $J_\gamma=J$ for short roots and
$J_\gamma=J_0$ for long roots.
\begin{lemma}[Reduction]\label{lem:cover-reduction}
Reduction induces a surjection $r_R:\Scover_R\to\Scover_k$ with kernel
\[
 D_R=\left\langle\!\left\langle
 \widetilde y_\gamma(t):t\in J_\gamma,\ \gamma\in G_2
 \right\rangle\!\right\rangle_{\Scover_R}.
\]
\end{lemma}
\begin{proof}
All residue parameters lift by Lemma~\ref{lem:local-cell}. Modulo the
displayed normal closure, a root generator depends only on its residue.
Lifting the defining relations gives a homomorphism from $\Scover_k$
to this quotient, inverse to the map induced by reduction.
\end{proof}
Write $\U_C(J)=\ker(r_R|_{\U_C})=\prod_{\gamma\in C}
\widetilde y_\gamma(J_\gamma)$ in any fixed order.
We distinguish the three kernels
\[
 K_R=\ker q_R,\qquad D_R=\ker r_R,\qquad
 N_R=\ker(\pi_R:E_R\to E_k).
\]
They satisfy $q_kr_R=\pi_Rq_R$ and
$r_R^{-1}(K_k)=q_R^{-1}(N_R)$. No normal-generation description
of $N_R$ is assumed.

\subsection{Toral action and rank-one lifting}\label{subsec:lifted-levi}
Put $H_R=T_Q(R)^{\tau\rho}$ and
\[
 \begin{aligned}
 P_\gamma&=\langle\widetilde y_\gamma(R_\gamma),
                    \widetilde y_{-\gamma}(R_\gamma)\rangle,\\
 \Llift_R&=\langle P_\gamma\cap q_R^{-1}(H_R):\gamma\in G_2\rangle,
 \qquad L_R^1=\Llift_R\cap D_R.
 \end{aligned}
\]
Thus $\Llift_R$ is generated by rank-one words with toral concrete
image; it is not defined as the whole inverse image of $H_R$.
For $h\in H_R$, let $m_\gamma(h)\in R_\gamma^\times$ be the multiplier
in $hy_\gamma(t)h^{-1}=y_\gamma(m_\gamma(h)t)$.

\begin{lemma}[Toral action]\label{lem:levi}
For $l\in\Llift_R$, $\gamma\in G_2$, and $t\in R_\gamma$, one has
\[
 l\widetilde y_\gamma(t)l^{-1}
 =\widetilde y_\gamma(m_\gamma(q_R(l))t).
\]
Hence $\Llift_R$ normalizes $\U_C$ and $\U_C(J)$ for every chamber.
If $l\in L_R^1$, then $m_\gamma(q_R(l))\equiv1\pmod{J_\gamma}$.
\end{lemma}
\begin{proof}
First take $l\in P_\alpha$ with toral image and $\gamma$ noncollinear
with $\alpha$. Choose a functional $\lambda$ with
$\lambda(\alpha)=0$ and $\lambda(\gamma)>0$. Its positive sector lies
in a chamber. Conjugating a sector root $\beta$ by a root at
$\pm\alpha$ introduces only roots $i(\pm\alpha)+j\beta$, $j>0$,
still in the sector. Thus $P_\alpha$ normalizes the lifted sector.
Both sides of the asserted formula lie there and have the same
concrete image, so unipotent injectivity gives equality.

For the $\alpha$-line use the off-axis identities
\[
 \begin{aligned}
 \widetilde y_c(t)&=
 [\widetilde y_a(1),\widetilde y_b(t)]_R
 \widetilde y_f(-N(t))\widetilde y_e(-N(t))\widetilde y_d(t^\#),\\
 \widetilde y_f(v)&=[\widetilde y_a(v),\widetilde y_e(1)]_R.
 \end{aligned}
\]
Every letter lies off the output-root line, and all lie in one
unipotent group. Concrete Weyl conjugation lifts to a graded
automorphism of the cover \cite[Lemma~7.1]{EJK}, so these identities
give such an expression for every root. This does not assume a
conjugation formula for a chosen cover word. Apply the noncollinear
formula to each letter and use unipotent injectivity again.
The formula now holds for the generators of $\Llift_R$, hence for
their products. Multipliers preserve the radical levels, and
reduction at $t=1$ gives the final congruence.
\end{proof}

For a signed root $\gamma$, the absolute rank-one maps give
$j_\gamma:\SL_2(R_\gamma)\to\mathbf G_Q(R)$ with
$j_\gamma(x_\pm(t))=y_{\pm\gamma}(t)$, where
$x_+(t)=\left(\begin{smallmatrix}1&t\\0&1\end{smallmatrix}\right)$ and
$x_-(t)=x_+(t)^t$. For a short root this is the product of three
commuting absolute maps with the prescribed signs and powers of $\rho$.
Evaluating a chosen rank-one word in $\SL_2$ computes its concrete
image; no reverse map $P_\gamma\to\SL_2(R_\gamma)$ is assumed.

\begin{lemma}[Rank-one lifting]\label{lem:wg}\label{subsec:rank-one-correction}
A chosen rank-one word with toral concrete image has diagonal
$\SL_2(R_\gamma)$ evaluation. If $w\in P_\gamma$ and
$q_k(r_R(w))\in H_k$, then
\[
 w=\widetilde y_\gamma(j_+)\,l\,\widetilde y_{-\gamma}(j_-),
 \qquad j_\pm\in J_\gamma,\quad l\in P_\gamma\cap\Llift_R,
 \quad r_R(l)=r_R(w).
\]
In particular, $r_R(\Llift_R)=\Llift_k$.
\end{lemma}
\begin{proof}
Let $M=\left(\begin{smallmatrix}a&b\\c&d\end{smallmatrix}\right)$ be
the evaluation. On the Cartan element of the absolute root carrying
the free parameter, the other factors in a short orbit act trivially,
and
\[
 M\begin{pmatrix}1&0\\0&-1\end{pmatrix}M^{-1}
 =\begin{pmatrix}ad+bc&-2ab\\2cd&-ad-bc\end{pmatrix}.
\]
A toral image fixes this element. Independence of root and coroot
vectors gives $ad+bc=1$, $ab=cd=0$. Together with $ad-bc=1$ and
$2\in R^\times$, this yields $ad=1$, $bc=0$, and then $b=c=0$.

If the word is toral only after reduction, the same test gives
$b,c\in J_\gamma$ and $d\in R_\gamma^\times$. The identity
$M=x_+(b/d)\operatorname{diag}(d^{-1},d)x_-(c/d)$ shows that
\[
 l=\widetilde y_\gamma(-b/d)\,w\,\widetilde y_{-\gamma}(-c/d)
\]
has toral image. Rearranging gives the factorization with
$j_+=b/d$ and $j_-=c/d$. Since these parameters are radical,
$r_R(l)=r_R(w)$. Every defining generator of $\Llift_k$ lifts
root by root and can be corrected this way; products give
$r_R(\Llift_R)=\Llift_k$.
\end{proof}

\subsection{The congruence decomposition}\label{subsec:cover-congruence}
For $u,v\in R_\gamma$ with $\Delta=1+uv\in R_\gamma^\times$, define
\[
 \ell_\gamma(u,v)=\widetilde y_\gamma(-u/\Delta)
 \widetilde y_{-\gamma}(v)\widetilde y_\gamma(u)
 \widetilde y_{-\gamma}(-v/\Delta).
\]
Its rank-one evaluation is $\operatorname{diag}(\Delta^{-1},\Delta)$,
so $\ell_\gamma(u,v)\in\Llift_R$. Its definition and root additivity give
\begin{align}
 \widetilde y_{-\gamma}(v)\widetilde y_\gamma(u)
 &=\widetilde y_\gamma(u/\Delta)\ell_\gamma(u,v)
                  \widetilde y_{-\gamma}(v/\Delta),\label{eq:cover-exchange}\\
 \widetilde y_{-\gamma}(v)\widetilde y_\gamma(u)\widetilde y_{-\gamma}(-v)
 &=\widetilde y_\gamma(u/\Delta)\ell_\gamma(u,v)
                  \widetilde y_{-\gamma}(-uv^2/\Delta).
 \label{eq:cover-radical-conjugation}
\end{align}
If either input is radical, then $\ell_\gamma(u,v)\in L_R^1$.
If $u$ is radical, both outer parameters in the second identity
are radical, even when $v$ is not.

\begin{proposition}[Congruence decomposition]\label{prop:congruence}\label{prop:cover-cell}
For every chamber $C$,
$D_R=\U_C(J)L_R^1\U_{-C}(J)$.
\end{proposition}
\begin{proof}
Write $B_C=\U_C(J)L_R^1\U_{-C}(J)\subseteq D_R$.
For a simple root $\gamma\in C$, arbitrary-order unipotent coordinates
give
\[
 B_C=V^+\widetilde y_\gamma(J_\gamma)L_R^1
       \widetilde y_{-\gamma}(J_\gamma)V^-,
 \qquad
 V^\pm=\langle\widetilde y_{\pm\beta}(J_\beta):
                         \beta\in C\setminus\{\gamma\}\rangle.
\]
Move the middle factors by Lemma~\ref{lem:levi} and interchange the
opposite root factors using \eqref{eq:cover-exchange}. All denominators
lie in $1+J_\gamma$, and the new toral words lie in $L_R^1$.
Since $s_\gamma C=(C\setminus\{\gamma\})\cup\{-\gamma\}$, the reverse
interchange gives $B_C=B_{s_\gamma C}$. Connectivity of the chamber
graph makes all $B_C$ one set $B$.

Every radical root subgroup stabilizes $B$ by left multiplication:
choose a chamber containing that root. The same holds for $L_R^1$
by its toral action, and for inverses. Since $1\in B$, the set $B$
is exactly the subgroup generated by these factors.

This subgroup is normal in $\Scover_R$. Nonopposite conjugation of
a radical root gives radical root factors: each correction contains
the radical input, $\rho(J)=J$, and $T(J)\subseteq J_0$.
Opposite conjugation is covered by \eqref{eq:cover-radical-conjugation}.
For $l\in L_R^1$, Lemma~\ref{lem:levi} gives
\[
 \widetilde y_\beta(t)l\widetilde y_\beta(-t)
 =\widetilde y_\beta\bigl((1-m_\beta(q_R(l)))t\bigr)l\in B,
\]
since the parameter is radical. The same arguments apply to inverse
conjugators, proving normality. As $B$ contains every radical root,
Lemma~\ref{lem:cover-reduction} gives $D_R\subseteq B$.
\end{proof}

\subsection{Field comparison and centrality}\label{subsec:field-comparison}
Only the next lemma uses a reductive interpretation over the fixed
field; no such interpretation over $R_0$ is needed.

\begin{lemma}[The field kernel]\label{lem:field-placement}
For a field $k$ of characteristic different from $2,3$, with
$\bar\rho\in\Aut(k)$ and $\bar\rho^3=\id$, one has
$K_k\subseteq\Llift_k$.
\end{lemma}
\begin{proof}
Put $F=k^{\bar\rho}$. If $\bar\rho$ has order three, descent of the
pinned $\operatorname{Spin}_{8,k}$ along $\tau\bar\rho$ over the cyclic
cubic Galois extension $k/F$ gives a simply connected quasi-split
group $\mathbf H/F$. Its relative root system is $G_2$, with short
parameters $k$ and long parameters $F$; its root maps become our
orbit products over $k$.

If $\bar\rho=\id$, take the pinned fixed group
$\mathbf H=\operatorname{Spin}_{8,k}^{\tau}$. The fixed-pinning results
\cite[Theorem~5.1 and Corollary~5.8]{ALRR} make it connected, smooth,
and split reductive: the coinvariants of the permuted fundamental
weight basis are free, and no exceptional even-$A$ component occurs.
The fixed-root construction \cite[\S4.3.1 and Proposition~6.1]{ALRR}
gives simple roots $(2,-1)$ and $(-3,2)$ in the coinvariant basis
of outer and middle weights. Thus $\mathbf H$ is split $G_2$,
both simply connected and adjoint, since this matrix has determinant
one. The adjoint $D_4$ fixed group has the same pinned root datum.

In either case, the central isogeny
$\operatorname{Spin}_8(k)\to\mathbf G_Q(k)$ is injective on
$\mathbf H(F)$: its kernel $\mu_2^2(k)$ has no nonidentity
$\tau\bar\rho$-fixed point, since $\tau$ cycles its three nonidentity
points and $\bar\rho$ fixes them. For the standard minimal parabolic
$P$ of $\mathbf H$, the elementary subgroup $E_P(F)$ therefore maps
isomorphically onto $E_k$, root by root. The inverse image of the
absolute adjoint torus is the absolute Spin torus; intersecting with
$\mathbf H(F)$ gives its torus Levi subgroup $T_F(F)$.
Thus the rank-one toral-image conditions agree.

The relative Steinberg presentation
\cite[Definition~3.1]{Stavrova} consists here of additive root laws
and the same nonopposite relations. In the nonsplit case this follows
after extension to $k$ by absolute unipotent coordinates; in the split
case it follows from the pinned orbit maps. Lemma~\ref{lem:borel}
identifies $\Scover_k$ with $\operatorname{St}_P(F)$ and
$\Llift_k$ with $\widetilde L_P(F)$ of
\cite[Definition~3.9]{Stavrova}. Apply
\cite[Lemma~3.15]{Stavrova} over the field $F$, viewed as a local
ring with zero maximal ideal. The group is simply connected and
simple of relative rank two; the congruence factor is trivial,
leaving
$\ker(\operatorname{St}_P(F)\to E_P(F))\subseteq\widetilde L_P(F)$.
No lower bound on $|F|$ or perfectness assumption is required.
\end{proof}

\begin{theorem}[Centrality]\label{thm:centrality}
For every local coefficient pair,
\[
 K_R=Z(\Scover_R),\qquad K_R\subseteq\Llift_R,\qquad
 \Scover_R/Z(\Scover_R)\simeq E_R.
\]
\end{theorem}
\begin{proof}
For $z\in K_R$, reduction gives $r_R(z)\in K_k$.
Lemmas~\ref{lem:field-placement} and~\ref{lem:wg} supply
$l\in\Llift_R$ with $r_R(l)=r_R(z)$. Thus $l^{-1}z\in D_R$.
Fix opposite chambers and write their lifted groups as $\U_\pm$.
Proposition~\ref{prop:congruence} and the toral action give
\[
 z=u_+hu_-,\qquad u_\pm\in\U_\pm(J),\quad h\in\Llift_R.
\]
Under $q_R$, this is an absolute open-cell decomposition of the
identity. Uniqueness forces $q_R(u_\pm)=q_R(h)=1$.
Unipotent injectivity gives $u_+=u_-=1$, so $z=h\in\Llift_R$.
Lemma~\ref{lem:levi} then makes $z$ commute with every lifted root,
since $m_\gamma(q_R(z))=1$. Hence $K_R\subseteq Z(\Scover_R)$.
Conversely, surjectivity of $q_R$ sends the center into $Z(E_R)=1$
by Lemma~\ref{lem:ambient-center}, proving equality.
\end{proof}

\subsection{Coefficient substitutions}\label{sec:converse}
A unital ring homomorphism $\eta:S\to A$ is \emph{trace-preserving}
if $\eta T_S=T_A\eta$. We first construct its map on graded covers,
then show when it descends to elementary groups.

\begin{lemma}\label{lem:trace-polynomials}
If $2,3\in A^\times$, a trace-preserving map $\eta:S\to A$ preserves
$\#$, $N$, and $\times$, and sends $S^{\rho_S}$ into $A^{\rho_A}$.
\end{lemma}
\begin{proof}
The integral identity
\[
 2t^\#=T(t)^2-T(t^2)-2tT(t)+2t^2
\]
gives $\eta(t^\#)=\eta(t)^\#$ after applying $\eta$ and cancelling
$2$ in $A$. Multiplication gives preservation of $N$, and polarization
gives $\times$. If $v\in S^{\rho_S}$, then
$T_A(\eta(v))=3\eta(v)$. Applying $\rho_A$ and cancelling $3$ gives
$\rho_A(\eta(v))=\eta(v)$. No division in $S$ is used.
\end{proof}

Every coefficient in \eqref{eq:ab}--\eqref{eq:cd} is formed from
these operations and ring operations; Weyl transport adds only signs.
Thus $\eta$ preserves all cover relations, and Lemma~\ref{lem:borel}
gives
\[
 F_\eta:\Scover_S\to\Scover_A,\qquad
 F_\eta(\widetilde y_{\gamma,S}(t))
       =\widetilde y_{\gamma,A}(\eta(t)).
\]
These maps respect composition. They descend to a unique
$\phi_\eta:E_S\to E_A$ precisely when $F_\eta(K_S)\subseteq K_A$.

\begin{theorem}[Coefficient automorphisms]\label{thm:local-converse}
For a local pair $(R,\rho)$, every $\eta\in\Aut(R)$ with $\eta T=T\eta$
induces an automorphism $\phi_\eta$ of $E_R$ by
$y_\gamma(t)\mapsto y_\gamma(\eta(t))$.
\end{theorem}
\begin{proof}
The inverse $\eta^{-1}$ preserves trace and induces the inverse of
$F_\eta$. By Theorem~\ref{thm:centrality},
$K_R=Z(\Scover_R)$ is characteristic, so both maps descend to $E_R$.
\end{proof}

\begin{proposition}[A local source]\label{prop:local-source}
Let $(S,\rho_S)$ be a local coefficient pair and $(A,\rho_A)$ any
coefficient pair with $2,3\in A^\times$. Every unital trace-preserving
map $\eta:S\to A$ induces the root substitution
$\phi_\eta:E_S\to E_A$.
\end{proposition}
\begin{proof}
For $z\in K_S$, source centrality makes $F_\eta(z)$ commute with every
lifted target parameter-one root, since $\eta(1)=1$.
Its concrete image centralizes the target frame and is therefore
trivial by Lemma~\ref{lem:ambient-center}.
Hence $F_\eta(K_S)\subseteq K_A$. Neither target centrality nor
surjectivity of $F_\eta$ is used.
\end{proof}

\begin{lemma}[Orientation over a domain]\label{lem:orientation}
If $A$ is a domain with $2\ne0$, every unital trace-preserving map
$\eta:S\to A$ satisfies
$\eta\rho_S=\rho_A\eta$ or $\eta\rho_S=\rho_A^{-1}\eta$.
In particular, $\ker\eta$ is $\rho_S$-stable.
\end{lemma}
\begin{proof}
For $t\in S$, put
$a=\eta(\rho_S(t))$, $b=\eta(\rho_S^2(t))$,
$c=\rho_A(\eta(t))$, $d=\rho_A^2(\eta(t))$.
Trace preservation on $t,t^2$ gives
$a+b=c+d$ and $a^2+b^2=c^2+d^2$. Cancelling $2$ yields $ab=cd$,
hence $(a-c)(a-d)=0$. Since $A$ is a domain, either $a=c$ or $a=d$.
The two solution sets are additive subgroups covering $S$.
A group cannot be the union of two proper subgroups, so one orientation
holds on all of $S$; either makes $\ker\eta$ stable under $\rho_S$.
\end{proof}

\begin{theorem}[A local target]\label{thm:arbitrary-converse}
For an arbitrary coefficient pair $(S,\rho_S)$ and a local pair
$(A,\rho_A)$, every unital trace-preserving map $\eta:S\to A$ induces
a unique homomorphism $\phi_\eta:E_S\to E_A$ with the coherent root
rule. The source need not contain $1/2$ or $1/3$.
\end{theorem}
\begin{proof}
Compose with $A\to k=A/\Jac(A)$ to obtain $\nu:S\to k$.
It preserves trace, so Lemma~\ref{lem:orientation} gives one orientation
for $\nu$. Thus $\mathfrak p=\ker\nu=\eta^{-1}(\Jac(A))$ is a
$\rho_S$-stable prime, and $2,3\notin\mathfrak p$.
The ring $B=S_{\mathfrak p}$ inherits $\rho_S$ and is a local
coefficient pair. Every denominator maps to a unit of $A$, so
$\eta$ extends to $\eta_{\mathfrak p}:B\to A$.

For $s\notin\mathfrak p$, stability gives
$N_S(s)\in S^{\rho_S}\setminus\mathfrak p$, and
\[
 T_B(t/s)=\frac{T_S(ts^\#)}{N_S(s)}.
\]
Lemma~\ref{lem:trace-polynomials}, applied to $\eta$, gives
\[
 \eta_{\mathfrak p}(T_B(t/s))
 =\frac{T_A(\eta(t)\eta(s)^\#)}{N_A(\eta(s))}
 =T_A\bigl(\eta(t)\eta(s)^{-1}\bigr),
\]
since the norm denominator is a fixed unit. Thus $\eta_{\mathfrak p}$
preserves trace. Apply Proposition~\ref{prop:local-source} to $B\to A$
and compose with the equivariant pinned localization $E_S\to E_B$.
The resulting map has the required rule and is unique by root
generation. No centrality of $K_S$ or surjectivity of elementary
localization is used. The orientation concerns $\nu$, not necessarily
$\eta$.
\end{proof}

The equivariant and anti-equivariant pinned constructions of
Section~\ref{sec:residual} agree with these substitutions on root
generators. Lemma~\ref{lem:orientation} therefore also gives a
substitution for a domain target with $2\ne0$, without requiring
$2$ or $3$ invertible. Whenever substitutions exist,
$\phi_\theta\circ\phi_\eta=\phi_{\theta\circ\eta}$, and injectivity
of a short-root map makes $\eta\mapsto\phi_\eta$ injective.
No general converse for two arbitrary nonlocal coefficient rings
is asserted.

Together with Theorem~\ref{thm:relative}, this proves
Theorem~\ref{thm:main-B}.

%-----------------------
%-----------------------
%-----------------------

\medskip
\section{Automorphism classification}\label{sec:classification}

Let $(R,\rho)$ be a local coefficient pair. We prove that the reduction
kernel is the greatest proper normal subgroup, apply the relative
theorem, and pass to the full groups and both forms.

\subsection{The residue kernel}\label{sec:radical}

Put $\pi_R:E_R\to E_k$ and $N_R=\ker\pi_R$.
Throughout this subsection, $H\trianglelefteq E_R$; the fixed torus
is written $T_Q(R)^\delta$.

\begin{lemma}\label{lem:simple-residue}
The map $\pi_R$ is surjective and $E_k$ is nontrivial simple.
\end{lemma}
\begin{proof}
Root parameters lift by Lemma~\ref{lem:local-cell}, proving
surjectivity, and injectivity of the long-root map gives
$y_{a,k}(1)\ne1$.
If $\bar\rho=\id$, Lemma~\ref{lem:field-placement} identifies $E_k$
with split $G_2(k)$. The simplicity theorem
\cite[Theorem~5, p.~47]{Steinberg} applies: the center is trivial
and the small exceptional fields are excluded.

If $\bar\rho$ has order three, the same root comparison gives a
surjection from the universal root-generated triality group
$G'_\sigma=\langle U_\sigma,U^-_\sigma\rangle$ of
\cite[\S11, pp.~177--178]{Steinberg} to $E_k$.
Its central quotient is simple by Theorem~34 on p.~187 of that work,
with no exception of type ${}^3D_4$.
Its center maps into $Z(E_k)=1$ by Lemma~\ref{lem:ambient-center}.
Thus $E_k$ is a nontrivial image of a simple group, and is simple.
\end{proof}

\begin{lemma}[A unit root]\label{lem:unit-generation}
If $H$ contains a short or long root element with a unit parameter,
then $H=E_R$.
\end{lemma}
\begin{proof}
Use right commutators. Transport a long unit to $y_a(v)$,
$v\in R_0^\times$. Relation~\eqref{eq:ae} gives every $y_f(vw)$,
$w\in R_0$, and Weyl transport gives all long root groups.
Relation~\eqref{eq:ab} at $v=1$, after deleting its last two
long factors, gives
$Q(t)=y_c(t)y_d(-t^\#)\in H$ for every $t\in R$.
Since $(-t)^\#=t^\#$, adjacent cancellation gives
\[
 Q(t/2)Q(-t/2)^{-1}
 =y_c(t/2)y_d(-(t/2)^\#)y_d((t/2)^\#)y_c(t/2)=y_c(t).
\]
Weyl transport gives all short root groups.
If the initial unit is short, transport it to $y_d(u)$ and use
\[
 \commR{y_d(u)}{y_b(-u^{-1}/3)}=y_e(1)
\]
from \eqref{eq:db}, reducing to the long case.
\end{proof}

Write $U^\pm$ for the relative unipotent groups and
$U_0=\langle y_\gamma(R_\gamma):\gamma\in\{b,c,d,e,f\}\rangle$.
Positive coordinates allow the simple long root $a$ to be placed first.

\begin{lemma}[Extraction]\label{lem:extraction}
Let $H\trianglelefteq E_R$.
\begin{enumerate}
\item If a product of factors from distinct positive root subgroups
belongs to $H$, then every factor belongs to $H$, regardless of order.
\item If $y_a(t)Xhu^-\in H$, where $t\in R_0$, $X\in U_0$,
$h\in T_Q(R)^\delta$, and $u^-\in U^-$, then $y_a(t)\in H$.
\end{enumerate}
\end{lemma}
\begin{proof}
Appendix~\ref{app:extraction} proves both assertions by explicit
commutator words. The first uses increasing root height and deletion
of extracted factors in their original positions. The second produces
a leading short parameter $-t$, removes the negative factors in two
steps, and applies positive-factor extraction.
Only normality under $E_R$ is used; the individual torus and unipotent
factors need not belong to $H$.
\end{proof}

\begin{theorem}\label{thm:radical}
The subgroup $N_R$ is the greatest proper normal subgroup of $E_R$.
It is characteristic, and reduction gives an exact sequence
\[
 1\longrightarrow N_R\longrightarrow E_R
 \xrightarrow{\pi_R}E_k\longrightarrow1.
\]
\end{theorem}
\begin{proof}
If $H\not\subseteq N_R$, simplicity of $E_k$ gives $\pi_R(H)=E_k$.
Choose $z\in H$ with residue $y_{a,k}(1)$.
This residue lies in the absolute open cell, so
Lemma~\ref{lem:local-cell} gives a unique relative decomposition
$z=u^+hu^-$ with
$\overline{u^+}=y_{a,k}(1)$ and $\bar h=\overline{u^-}=1$.
Write $u^+=y_a(t)X$, $X\in U_0$. Positive coordinates give
$t\in1+J_0\subseteq R_0^\times$.
Lemma~\ref{lem:extraction} places $y_a(t)$ in $H$, and
Lemma~\ref{lem:unit-generation} gives $H=E_R$.

Thus every proper normal subgroup lies in $N_R$, which is proper
by Lemma~\ref{lem:simple-residue}. This characterization makes $N_R$
characteristic; exactness follows from surjectivity of reduction.
\end{proof}

This proof is independent of cover centrality. It neither identifies
$N_R$ with the normal closure of the radical root elements nor
classifies normal subgroups by invariant ideals.

\subsection{Automorphisms of the elementary group}
\label{sec:local-classification}

For $\alpha\in\Aut(E_R)$, Theorem~\ref{thm:radical} defines
$\bar\alpha\in\Aut(E_k)$ by
$\bar\alpha(\pi_R(x))=\pi_R(\alpha(x))$.
Write $\kappa_R:R\to k$ for coefficient reduction.

\begin{lemma}[The field factor]\label{lem:field-factor}
There are $\bar g\in G_Q(k)$, $\nu\in\Aut(k)$, and
$\epsilon\in\{1,-1\}$ such that
\begin{equation}
 \nu\bar\rho=\bar\rho^\epsilon\nu,\qquad
 \bar\alpha=\Int(\bar g)\circ\Phi_\nu^\epsilon.
 \label{eq:local-residual-factor}
\end{equation}
The element $\bar g$ lifts to $G_Q(R)$. Thus $\alpha$ satisfies
\eqref{eq:typed-rs} with coefficient map $\nu\circ\kappa_R$,
quotient $k$, and group image $H=E_k$.
\end{lemma}
\begin{proof}
If $\bar\rho=\id$, use the pinned $G_2$ comparison and
\cite[Theorem~3, p.~16]{BuninaG2}, which applies to fields with
$1/2,1/3$ without perfectness.
If $\bar\rho$ has order three, use the equal-order triality
classification \cite[Theorem~36, p.~195]{Steinberg}, with the field
and root conventions on pp.~158--159 and 177--178.
It supplies inner, diagonal, and field factors.

In the adjoint form, a diagonal factor is a torus character
$\chi:Q\to k^\times$. Compatibility with the triality root products
gives $\chi(\tau\alpha_i)=\bar\rho(\chi(\alpha_i))$:
the outer scalings form a $\bar\rho$-orbit, and the middle scaling
is fixed. Hence the inner and diagonal factors combine into
$\bar g\in G_Q(k)$.
The remaining map $\beta=\Int(\bar g^{-1})\circ\bar\alpha$ fixes
the four simple signed constants: field factors fix $\pm1$,
and pinned diagram identifications permute the commuting outer
constant factors and fix the middle one.
The Tits words then give the whole fixed frame.

Theorem~\ref{thm:recovery}, applied directly to this frame-fixing
automorphism over $k$, gives $\nu\in\Aut(k)$ preserving trace.
Lemma~\ref{lem:orientation} supplies one orientation for $\nu$.
The pinned map $\mathbf G_Q(\nu)$ or
$d_0\circ\mathbf G_Q(\nu)$ from Section~\ref{sec:residual}
agrees with $\beta$ on every root, proving
\eqref{eq:local-residual-factor}.

Proposition~\ref{prop:factorization} lifts $\bar g$.
Since $\ker(\nu\circ\kappa_R)=J$, equivariance of reduction
and surjectivity of $\pi_R$ give the data in \eqref{eq:typed-rs}.
No lift of $\nu$ to $R$ has been assumed.
\end{proof}

\begin{theorem}[Elementary automorphisms]\label{thm:local-adjoint}
Every $\alpha\in\Aut(E_Q(R))$ has a unique expression
\[
 \alpha=\Int(g)\circ\phi_\eta,\qquad
 g\in G_Q(R),\quad \eta\in\Aut_T(R,\rho).
\]
Conversely, every such pair defines an automorphism.
\end{theorem}
\begin{proof}
Lemma~\ref{lem:field-factor} permits application of
Theorem~\ref{thm:relative}, giving $g$ and a unital trace-preserving
ring map $\eta$.
Normality of $E_R$ makes $\psi=\Int(g^{-1})\circ\alpha$ a
frame-fixing automorphism of $E_R$.
Theorem~\ref{thm:recovery} makes $\eta$ invertible, and its root
rule identifies $\psi$ with $\phi_\eta$.
The pair is unique by Theorem~\ref{thm:relative}.

Conversely, Theorem~\ref{thm:local-converse} constructs $\phi_\eta$,
and $g\in G_Q(R)$ normalizes $E_Q(R)$.
\end{proof}

For the field data above, Theorem~\ref{thm:relative} also gives
$\operatorname{red}_R(g)=\bar g$ and
$\kappa_R\circ\eta=\nu\circ\kappa_R$.
Thus the coefficient lift is recovered from the group automorphism.
The orientation obtained here concerns the residue field;
the local condition remains trace preservation.

\subsection{The two forms and the full groups}\label{sec:forms}

We compare the two forms and extend the coefficient automorphisms,
using the factorization
$G_\Lambda(R)=E_\Lambda(R)H_\Lambda$,
$H_\Lambda=T_\Lambda(R)^\delta$, and the equality
$G_P(R)=E_P(R)$ from Proposition~\ref{prop:factorization}.

\begin{proposition}\label{prop:forms}
The central isogeny $\pi:\mathbf G_P\to\mathbf G_Q$ is injective on
$G_P(R)$ and induces an isomorphism $\pi_E:E_P(R)\to E_Q(R)$ with
$\pi_E(y_{\gamma,P}(t))=y_{\gamma,Q}(t)$.
Every $g\in G_Q(R)$ has a canonical adjoint action
$\operatorname{Ad}_P(g)$ on $G_P(R)$ preserving $E_P(R)$,
compatible with conjugation under $\pi$.
\end{proposition}
\begin{proof}
If $u^2=1$, then $(1-u)/2$ is an idempotent in the local ring $R$,
so $u=\pm1$. Thus $\rho$ fixes all four points of the central kernel
$\mu_2^2(R)$, while $\tau$ cycles the three nonidentity points.
The fixed kernel is trivial, proving injectivity.
The relative root words lift every generator of $E_Q(R)$,
proving surjectivity of $\pi_E$.

Conjugation on $\mathbf G_P$ factors through the central quotient
$\mathbf G_Q$ \cite[Theorem~7.1.9]{Conrad}: lifts after a faithfully
flat extension differ centrally and define the same action.
This action is $\tau\rho$-equivariant, so fixed points act on fixed
points. Normality of $E_Q(R)$ and injectivity of $\pi$ on $G_P(R)$
show that it preserves $E_P(R)$. No lift of $g$ over $R$ is required.
\end{proof}

In particular, $\pi(G_P(R))=E_Q(R)$.
A $\delta$-fixed lift of $g\in G_Q(R)$ exists exactly when
$g\in E_Q(R)$, and is then unique.

\smallskip\noindent\emph{The torus action.}
\label{subsec:torus-substitution}
The simple-root characters identify $H_Q$ with
$R^\times\times R_0^\times$. Write $h=t(u,v)$ for the point with
\[
 \alpha_4(h)=u,\quad \alpha_1(h)=\rho(u),\quad
 \alpha_3(h)=\rho^2(u),\quad \alpha_2(h)=v.
\]
Its multipliers on the coherent root maps are
\[
 \begin{array}{c|cccccc}
 \gamma&b&a&c&d&e&f\\\hline
 m_\gamma(t(u,v))&u&v&uv&u^\#v&N(u)v&N(u)v^2
 \end{array}
\]
and $m_{-\gamma}=m_\gamma^{-1}$.
In the $d$-map the free parameter occurs at
$\alpha_1+\alpha_2+\alpha_3$, giving the multiplier $u^\#v$.

For $\eta\in\Aut_T(R,\rho)$, Lemma~\ref{lem:trace-polynomials}
applied to $\eta$ and $\eta^{-1}$ gives $\eta(R_0)=R_0$ and
preservation of $\#$ and $N$. Hence
\[
 \eta_T(t(u,v))=t(\eta(u),\eta(v))
\]
is a torus automorphism satisfying
$\eta(m_\gamma(h))=m_\gamma(\eta_T(h))$.
Checking root generators gives, for $h\in H_Q$ and $x\in E_Q(R)$,
\begin{equation}
 \phi_\eta(hxh^{-1})=\eta_T(h)\phi_\eta(x)\eta_T(h)^{-1}.
 \label{eq:forms-torus-compatibility}
\end{equation}
This substitutes in the independent torus coordinates; it is not
an ambient coefficient action when $\eta$ does not commute with $\rho$.

\begin{proposition}[Extension]\label{prop:full-extension}
The automorphism $\phi_\eta$ extends uniquely to an automorphism
of $G_Q(R)$ by
\[
 \Phi_\eta(eh)=\phi_\eta(e)\eta_T(h)
 \qquad(e\in E_Q(R),\ h\in H_Q).
\]
On $G_P(R)=E_P(R)$, put
$\phi_{\eta,P}=\pi_E^{-1}\circ\phi_\eta\circ\pi_E$.
Then $\pi\circ\phi_{\eta,P}=\Phi_\eta\circ\pi$ on $G_P(R)$.
\end{proposition}
\begin{proof}
For $h\in E_Q\cap H_Q$, equation~\eqref{eq:forms-torus-compatibility}
shows that $\phi_\eta(h)$ and $\eta_T(h)$ induce the same conjugation
on $\phi_\eta(E_Q)=E_Q$. Lemma~\ref{lem:ambient-center} therefore gives
\begin{equation}
 \phi_\eta(h)=\eta_T(h)\qquad(h\in E_Q\cap H_Q).
 \label{eq:forms-intersection-agreement}
\end{equation}
If $eh=e'h'$, then $e'^{-1}e=h'h^{-1}$ lies in this intersection,
so the formula is well defined.
Equation~\eqref{eq:forms-torus-compatibility} applied to
$(eh)(e'h')=e(he'h^{-1})hh'$ proves multiplicativity.
The construction with $\eta^{-1}$ gives the inverse.

At each $g\in G_Q(R)$, two extensions have values inducing the same
conjugation on $\phi_\eta(E_Q)=E_Q$, since $E_Q$ is normal.
Their quotient centralizes $E_Q$, so is trivial.
This proves uniqueness. Compatibility with $\pi$ holds on the
root generators, hence on $G_P=E_P$.
\end{proof}

Evaluating the $\alpha_4,\alpha_2$ characters on the torus words
of Proposition~\ref{prop:factorization} gives
\begin{equation}
 \pi(h_{b,P}(r)h_{a,P}(s))
 =t\left(\frac{r^2}{s},\frac{s^2}{N(r)}\right),
 \qquad r\in R^\times,\ s\in R_0^\times.
 \label{eq:forms-torus-isogeny}
\end{equation}
The $P$ root-word rule sends $(r,s)$ to $(\eta(r),\eta(s))$
and preserves this formula because $\eta$ preserves $N$.

For an equivariant map $f:R\to A$ between local pairs and coefficient
automorphisms satisfying $f\eta_R=\eta_Af$, the corresponding full
actions commute with $f_*$. This follows on roots from the root rule,
on tori from $f_*(t(u,v))=t(f(u),f(v))$, and then by factorization.

\smallskip\noindent\emph{Perfection and completion.}

\begin{lemma}\label{lem:perfect}
One has $E_Q=[E_Q,E_Q]=[G_Q,G_Q]$.
Thus $E_Q$ is characteristic in $G_Q$, and $E_P=G_P$ is perfect.
\end{lemma}
\begin{proof}
The off-diagonal relative Cartan pairings are $-1,-3$, so the root
words $h_a(-1),h_b(-1)$ invert the other simple root parameters.
With left commutators,
\[
 \commL{h_a(-1)}{y_b(-t/2)}=y_b(t),\qquad
 \commL{h_b(-1)}{y_a(-v/2)}=y_a(v)
 \quad(t\in R,\ v\in R_0).
\]
Weyl transport makes every root generator a commutator in $E_Q$.
Since $G_Q=E_QH_Q$ and $H_Q$ is abelian, $G_Q/E_Q$ is abelian,
giving $[G_Q,G_Q]=E_Q$. The assertion for $P$ follows from $\pi_E$.
\end{proof}

\begin{proof}[Completion of Theorem~\ref{thm:main-A}]
The elementary adjoint case is Theorem~\ref{thm:local-adjoint}.
Transport through $\pi_E$ gives the simply connected classification
with the canonical adjoint action, since $G_P=E_P$.

By Lemma~\ref{lem:perfect}, every $\alpha\in\Aut(G_Q)$ restricts
to $E_Q$. Write this restriction as $\Int(g)\circ\phi_\eta$.
Then $\Int(g^{-1})\circ\alpha$ extends $\phi_\eta$, so
Proposition~\ref{prop:full-extension} identifies it with $\Phi_\eta$.
Restriction gives uniqueness of $(g,\eta)$, and the preceding
constructions show that every pair defines an automorphism.
\end{proof}

Thus restriction $\Aut(G_Q)\to\Aut(E_Q)$ is an isomorphism.
Uniqueness in the $P$ form concerns the adjoint data, not a fixed lift.
A factorization $g=eh$ separates inner and diagonal actions, but
these factors need not be unique because $E_Q\cap H_Q$ may be
nontrivial. No central factor occurs: all four groups are centerless
by Lemma~\ref{lem:ambient-center} and the elementary comparison.

%-----------------------
%-----------------------
%-----------------------

\medskip
\section{Coefficient symmetries and examples}\label{sec:examples}

For a local coefficient pair, we compare trace preservation with
normalization of $\langle\rho\rangle$ and illustrate the obstruction
to a fixed lift.
When $\rho=\id$, one has $T=3\id$ and
$\Aut_T(R,\rho)=\Aut(R)$; the root formulas and classification
reduce to split $G_2$. Nontrivial residue action imposes a stronger
restriction.

\begin{proposition}[Orientation]\label{prop:local-orientation}
If $\bar\rho$ has order three, every $\eta\in\Aut_T(R,\rho)$ satisfies
exactly one of $\eta\rho=\rho\eta$ and $\eta\rho=\rho^{-1}\eta$.
Thus $\Aut_T(R,\rho)=N_{\Aut(R)}(\langle\rho\rangle)$.
\end{proposition}
\begin{proof}
Put $D=\eta\rho-\rho\eta$ and $E=\eta\rho-\rho^2\eta$.
The calculation in Lemma~\ref{lem:orientation}, using cancellation
of $2$ but not the domain property, gives $D(t)E(t)=0$ for all $t$.
Choose $t_0$ so that the residue of $\eta(t_0)$ is not fixed by
$\bar\rho$, and hence has three distinct conjugates.
Then $D(t_0)-E(t_0)$ is a unit. By locality, one of $D(t_0),E(t_0)$
is a unit and the other is zero. Polarization gives
\[
 D(t_0)E(t)+D(t)E(t_0)=0\qquad(t\in R),
\]
forcing $D=0$ or $E=0$. Both cannot hold, since $\eta$ is surjective
and $\rho\ne\rho^2$. Conversely, a normalizer permutes $\rho,\rho^2$
and therefore preserves $T$.
\end{proof}

The reversing branch can be empty. If $R$ is a domain,
Lemma~\ref{lem:orientation} gives the same normalizer description
even when the residue action is trivial. Nilpotents allow additional
coefficient symmetries.

\begin{example}[A square-zero ring]\label{ex:squarezero}
Let $R=\mathbb Q\oplus V$, $V=\mathbb Q^2$, with $V^2=0$ and
$(q,v)(q',v')=(qq',qv'+q'v)$. Define
\[
 A_0=\begin{pmatrix}0&-1\\1&-1\end{pmatrix},\qquad
 \rho(q,v)=(q,A_0v).
\]
Then $R_0=\mathbb Q$, $T(q,v)=(3q,0)$, and every unital ring
endomorphism has the form
$\eta_L(q,v)=(q,Lv)$, $L\in M_2(\mathbb Q)$.
All preserve trace and are invertible exactly when
$L\in\GL_2(\mathbb Q)$.
Subsection~\ref{subsec:squarezero-proof} proves directly that
\[
 E_Q(R)=G_Q(R)\simeq V_{14}\rtimes G_2(\mathbb Q),\qquad
 V_{14}\simeq W_7\otimes_{\mathbb Q}V,\qquad
 \operatorname{End}_{G_2(\mathbb Q)}(V_{14})=M_2(\mathbb Q),
\]
where $W_7$ is the seven-dimensional irreducible module.
In these multiplicity coordinates, $T_L=\id_{W_7}\otimes L$
gives the group endomorphism $(Z,g)\mapsto(T_LZ,g)$, with
$\det T_L=(\det L)^7$.

By contrast, the normalizer of $\langle A_0\rangle$ in
$\GL_2(\mathbb Q)$ is
$\mathbb Q[A_0]^\times\rtimes\langle J_1\rangle$, where
\[
 J_1=\begin{pmatrix}1&0\\1&-1\end{pmatrix},\qquad
 J_1A_0J_1^{-1}=A_0^2.
\]
Indeed, commuting matrices are $pI+qA_0$, of determinant
$p^2-pq+q^2$, and every reverser is a commuting matrix times $J_1$.
For $L=\operatorname{diag}(1,2)$,
\[
 LA_0L^{-1}=\begin{pmatrix}0&-1/2\\2&-1\end{pmatrix}
\]
is neither $A_0$ nor $A_0^2$. Nevertheless, $\eta_L$ induces an
automorphism of the full group. It fixes the frame, so uniqueness
prevents absorbing it into a different adjoint factor.
Here $T_L$ acts on the multiplicity space, not necessarily by
restriction of the ambient coefficient action.
\end{example}

\begin{example}[The fixed ring]\label{ex:no-etale-descent}
Let $R=\mathbb F_7[\varepsilon]/(\varepsilon^2)$ and
$\rho(\varepsilon)=2\varepsilon$.
Then $\rho$ has order three and trivial residue action, but
$R^\rho=\mathbb F_7$ and $\dim_{\mathbb F_7}R=2$.
Thus $R/R^\rho$ is not cubic \'etale, although all local results
above apply. Unlike Example~\ref{ex:squarezero}, every ring
automorphism has the form $\varepsilon\mapsto c\varepsilon$,
$c\in\mathbb F_7^\times$, and commutes with $\rho$.
\end{example}

\begin{example}[A point without a fixed lift]\label{ex:torus}
Take $R=\mathbb C[z]_{(z)}$ and $\rho(f(z))=f(\omega z)$, where
$\omega^3=1$ and $\omega\ne1$.
The point $h=t(1+z,1)\in H_Q$ reduces to the identity.
A lift in $G_P(R)$ would lie in $H_P$, since the inverse image
of the absolute torus is the Spin torus.
Formula~\eqref{eq:forms-torus-isogeny} would give
$u=r^2/s$ and $v=s^2/N(r)$ for $u=1+z$, $v=1$, with
$r\in R^\times$ and $s\in R_0^\times$.
Since $N(u)=N(r)^2/s^3$, this forces
\[
 r^2=uN(u)v^2=(1+z)(1+z^3)=(1+z)^2(z^2-z+1).
\]
The quadratic has two simple complex zeros, neither equal to $-1$.
The right-hand side has odd valuations there, so is not a square
in $\mathbb C(z)$ and cannot equal $r^2$.
Thus $h$ has no fixed lift and, since
$\pi(G_P(R))=E_Q(R)$, one has $h\notin E_Q(R)$.
The canonical adjoint action is therefore needed in the $P$
classification even for points congruent to the identity.
\end{example}

\subsection{The additional actions in the square-zero example}
\label{subsec:squarezero-proof}

We construct the actions in Example~\ref{ex:squarezero} directly,
without using the coefficient converse.
All tensor products in this subsection are over $\mathbb Q$.
The constant splitting of the square-zero extension gives
\[
 \mathbf G_Q(R)\simeq
 (\mathfrak g\otimes V)\rtimes\mathbf G_Q(\mathbb Q),
 \qquad \mathfrak g=\mathfrak{so}_8(\mathbb Q,J_8).
\]
The kernel is additive because $V^2=0$, and constants act by the
adjoint representation. Taking fixed points under $\tau\rho$ gives
\[
 G_Q(R)\simeq V_{14}\rtimes G_2(\mathbb Q),\qquad
 V_{14}=(\mathfrak g\otimes V)^{\tau_*\otimes A_0}.
\]
There are six length-three root-space orbits and one outer-coroot
orbit, each contributing a copy of $V$.
Fixed root and Cartan directions contribute zero because $A_0-I$
is invertible. Thus $\dim V_{14}=14$, with weights
$d,c,b,0,-b,-c,-d$, each of multiplicity two.

To identify the multiplicity coordinates, put
\[
 E_b=M_{\alpha_4}+M_{\alpha_1}+M_{\alpha_3},\qquad
 E_a=M_{\alpha_2},\qquad F_b=E_b^t,\qquad F_a=E_a^t.
\]
These act by Lie brackets. Define
\[
 w_d(v)=M_{e_2+e_3}\otimes A_0v
       +M_{e_1+e_4}\otimes A_0^2v+M_{e_1-e_4}\otimes v,
\]
and, retaining the argument $v$, set
\[
 \begin{gathered}
 w_c=F_bw_d,\qquad w_b=F_aw_c,\qquad w_0=F_bw_b,\\
 w_{-b}=\tfrac12F_bw_0,\qquad
 w_{-c}=F_aw_{-b},\qquad w_{-d}=F_bw_{-c}.
 \end{gathered}
\]
Direct brackets, using $A_0^2+A_0+I=0$, give
\[
 \begin{array}{c|rrrrrrr}
 &w_d&w_c&w_b&w_0&w_{-b}&w_{-c}&w_{-d}\\\hline
 E_b&0&w_d&0&2w_b&w_0&0&w_{-c}\\
 F_b&w_c&0&w_0&2w_{-b}&0&w_{-d}&0\\
 E_a&0&0&w_c&0&0&w_{-b}&0\\
 F_a&0&w_b&0&0&w_{-c}&0&0
 \end{array}
\]
and
\[
 w_0(v)=H_{\alpha_4}\otimes v
       +H_{\alpha_1}\otimes A_0v+H_{\alpha_3}\otimes A_0^2v,
 \qquad H_\alpha=[M_\alpha,M_{-\alpha}].
\]
Each $w_\lambda:V\to V_{14}$ is injective, as a free absolute
coefficient shows. Their distinct weights make the images independent,
so together they give all of $V_{14}$.

The simple signed root subgroups act by the finite exponentials
of these operators, whose nilpotence indices are three for short
roots and two for long roots.
Thus $\Xi(e_\lambda\otimes v)=w_\lambda(v)$ identifies
$V_{14}$ with $W_7\otimes V$, with the constant group acting only
on $W_7$.
The torus point $t(2,3)$ has distinct eigenvalues
$12,6,2,1,1/2,1/6,1/12$ on the seven weights.
Its spectral projections and the connected nonzero arrows in the
table prove irreducibility of $W_7$.
They also force a commuting endomorphism of $V_{14}$ to have
the same block $L\in\operatorname{End}(V)$ on every weight:
\[
 \operatorname{End}_{G_2(\mathbb Q)}(V_{14})=M_2(\mathbb Q),
 \qquad
 T_L=\Xi(\id_{W_7}\otimes L)\Xi^{-1},\qquad
 \det T_L=(\det L)^7.
\]

Let $\iota:V_{14}\to G_Q(R)$ be the additive kernel inclusion,
and write $y_\gamma(v)=\iota(r_\gamma(v))$ for radical short
parameters. Linearizing the root products gives
\[
 \begin{array}{c|rrrrrr}
 \gamma&b&-b&c&-c&d&-d\\\hline
 r_\gamma(v)&-w_b(v)&w_{-b}(v)&-w_c(v)&-w_{-c}(v)&w_d(v)&-w_{-d}(v).
 \end{array}
\]
Every nonzero weight space is therefore elementary.
The zero-weight space is supplied by the exact word
\[
 \iota(w_0(v))
 =y_{-b}(1)y_b(-v)y_{-b}(-1)y_b(v)y_{-b}(-v),
\]
which follows from
$\operatorname{Ad}(y_{-b}(1))w_b(v)=w_b(v)+w_0(v)+w_{-b}(v)$.
Constants are elementary as well, so $E_Q(R)=G_Q(R)$.

For every $L\in M_2(\mathbb Q)$, equivariance makes
$\Psi_L(Z,g)=(T_LZ,g)$ a group endomorphism.
The preceding table and root additivity give
\[
 \Psi_L(y_\gamma(q+v))=y_\gamma(q+Lv)
 \qquad(\gamma\text{ short},\ q\in\mathbb Q,\ v\in V);
\]
long-root parameters are rational and fixed.
The map is invertible exactly when $L$ is, and then agrees with
$\phi_{\eta_L}$ and $\Phi_{\eta_L}$.

The action is on the multiplicity decomposition, not in general
the restriction of $\id_{\mathfrak g}\otimes L$:
on $w_d(v)$ the latter gives coefficients
$LA_0v,LA_0^2v,Lv$, whereas $T_Lw_d(v)=w_d(Lv)$ has coefficients
$A_0Lv,A_0^2Lv,Lv$.
Thus the constructed group action requires no commutation of
$L$ with $A_0$.

%------------------
%------------------
%------------------

\medskip

\appendix
\section{Finite normalization data}\label{app:certificate}
\label{app:d4-section3-reproducibility}

We specify the matrices and selected equations used in
Section~\ref{sec:normalization}, retaining its local notation
$R$ and $J=\Jac(R)$. All coefficients are integral, and the linear
parametrizations require only $2$ invertible.

\subsection{Matrices and retained words}
\label{app:certificate-forms}\label{app:certificate-words}

Write $s=N_s$, $l=N_l$, $W_s=lslsl$, and $W_l=slsls$.
Since $W_s^2=H_s$ and $W_l^2=H_l$, the stabilizer equations
$X_sW_s=W_sX_s$ and $X_lW_l=W_lX_l$ first separate the sign
spaces and then identify entries up to sign. Their full solutions are
\[
 X_s=
 \begin{pmatrix}
 s_{16}&-s_{15}&0&0&s_{14}&-s_{13}&0&0\\
 -s_{12}&s_{11}&0&0&-s_{10}&s_9&0&0\\
 0&0&s_1&s_2&0&0&s_3&s_4\\
 0&0&s_5&s_6&0&0&s_7&s_8\\
 s_9&s_{10}&0&0&s_{11}&s_{12}&0&0\\
 s_{13}&s_{14}&0&0&s_{15}&s_{16}&0&0\\
 0&0&s_{17}&s_{18}&0&0&s_{19}&s_{20}\\
 0&0&s_{21}&s_{22}&0&0&s_{23}&s_{24}
 \end{pmatrix}
\]
and
\[
 X_l=
 \begin{pmatrix}
 l_3&0&0&l_4&l_1&0&0&l_2\\
 0&l_{12}&-l_{11}&0&0&l_{10}&-l_9&0\\
 0&-l_8&l_7&0&0&-l_6&l_5&0\\
 l_{15}&0&0&l_{16}&l_{13}&0&0&l_{14}\\
 l_1&0&0&l_2&l_3&0&0&l_4\\
 0&l_5&l_6&0&0&l_7&l_8&0\\
 0&l_9&l_{10}&0&0&l_{11}&l_{12}&0\\
 l_{13}&0&0&l_{14}&l_{15}&0&0&l_{16}
 \end{pmatrix}.
\]

The standard values are $U_s,U_l$. The condition
$X_l\equiv U_l\pmod J$ means
$l_3,l_7,l_{12},l_{16}\in1+J$, $l_{11}\in-1+J$, and
$l_i\in J$ for the remaining indices.
The normalization uses the following Weyl words:
\[
 \begin{array}{c|ccc}
 &R_s[3]&R_s[5]&R_s[8]\\\hline
 \text{word}&l^{-1}X_sl&H_lX_sH_l&slX_s(sl)^{-1}\\
 \text{standard value}&y_c(-1)&y_b(-1)&y_d(-1)
 \end{array}
\]
\[
 \begin{array}{c|ccc}
 &R_l[3]&R_l[5]&R_l[8]\\\hline
 \text{word}&sX_ls^{-1}&lsX_l(ls)^{-1}&H_sX_lH_s\\
 \text{standard value}&y_e(-1)&y_f(-1)&y_a(-1).
 \end{array}
\]
Besides the stabilizer equations, retain
\[
 X_sH_lX_s=H_l,\quad (X_sN_s)^3=I_8,\qquad
 X_lH_sX_l=H_s,\quad (X_lN_l)^3=I_8,
\]
\[
 X_lR_l[5]=R_l[5]X_l,\qquad
 [X_s,X_l]_L=R_s[3]R_s[8]R_l[3]R_l[5]^2.
\]
The inversion equations give $R_s[5]=X_s^{-1}$ and
$R_l[8]=X_l^{-1}$. Thus the remaining relations are the
vanishing of the polynomial matrices
\[
 \begin{aligned}
 \mathcal I_s&=X_sH_lX_s-H_l,
 &\mathcal K_s&=(X_sN_s)^3-I_8,\\
 \mathcal I_l&=X_lH_sX_l-H_s,
 &\mathcal K_l&=(X_lN_l)^3-I_8,\\
 \mathcal C&=X_lR_l[5]-R_l[5]X_l,\\
 \mathcal G&=X_sX_lR_s[5]R_l[8]
             -R_s[3]R_s[8]R_l[3]R_l[5]^2.
 \end{aligned}
\]
The standard mixed identity is
\[
 [y_b(1),y_a(1)]_L=y_c(-1)y_d(-1)y_e(-1)y_f(-2).
\]
These integral root-word identities are transported exactly by
Lemma~\ref{lem:tits-coherence}.
All matrix components below are indexed by $1,\ldots,8$ in the
original ordered basis.

\subsection{The long-root calculation}
\label{app:certificate-long}\label{subsec:d4-long-root-exhaustion}

The following fifteen entries, with the displayed signs, define
the scalar equations used in the elimination:
\[
 \begin{array}{c|c@{\qquad}c|c}
 \toprule
 \text{polynomial}&\text{matrix component}
 &\text{polynomial}&\text{matrix component}\\
 \midrule
 C_{15}&(\mathcal C)_{2,4}
 &I_1&(\mathcal I_l)_{1,1}\\
 C_{25}&(\mathcal C)_{2,8}
 &I_2&-(\mathcal I_l)_{4,4}\\
 C_{19}&-(\mathcal C)_{1,3}
 &I_{14}&(\mathcal I_l)_{2,6}\\
 C_{17}&(\mathcal C)_{3,5}
 &I_{11}&-(\mathcal I_l)_{2,7}\\
 C_{31}&-(\mathcal C)_{4,6}
 &I_{15}&-(\mathcal I_l)_{2,3}\\
 C_{32}&-(\mathcal C)_{8,6}
 &I_3&-(\mathcal I_l)_{3,3}\\
 I_8&-(\mathcal I_l)_{1,5}
 &K_{11}&-(\mathcal K_l)_{2,3}\\
 I_{10}&-(\mathcal I_l)_{4,8}
 &&\\
 \bottomrule
 \end{array}
\]

\begin{proposition}\label{prop:d4-long-root-form}
The near-standard solutions of the long-root relations are precisely
the matrices $X_l(p,z)$ with
$p\in1+J$, $z\in J$, and $p^2-p-z^2=0$, where
\[
 X_l(p,z)=
 \begin{pmatrix}
 1&0&0&0&0&0&0&0\\
 0&1&p&0&0&0&z&0\\
 0&p-1&1&0&0&z&0&0\\
 0&0&0&1&0&0&0&0\\
 0&0&0&0&1&0&0&0\\
 0&0&-z&0&0&1&1-p&0\\
 0&-z&0&0&0&-p&1&0\\
 0&0&0&0&0&0&0&1
 \end{pmatrix}.
\]
\end{proposition}

\begin{proof}
After cancellation of $2$, equations $C_{15}=C_{25}=0$ give
\[
 \begin{pmatrix}l_{11}&l_9\\l_9&l_{11}\end{pmatrix}
 \begin{pmatrix}l_2\\l_4\end{pmatrix}=0.
\]
The determinant is $1$ modulo $J$, so $l_2=l_4=0$.
Since $l_{11}$ is a unit, $C_{19}=0$ gives $l_6=l_9$;
then $C_{17}=0$ gives
\begin{equation}
 l_8l_{11}=l_6^2.
 \label{eq:long-elimination-product}
\end{equation}
Equations $C_{31}=C_{32}=0$ now give
\[
 \begin{pmatrix}l_6&-l_{11}\\-l_{11}&l_6\end{pmatrix}
 \begin{pmatrix}l_{13}\\l_{15}\end{pmatrix}=0.
\]
Its determinant is $-1$ modulo $J$, so $l_{13}=l_{15}=0$.

Next $I_8=0$ gives $2l_1l_3=0$, hence $l_1=0$.
Then $I_1=0$ gives $(1-l_3)(1+l_3)=0$, hence $l_3=1$.
Similarly, $I_{10}=I_2=0$ give $l_{14}=0$ and $l_{16}=1$.
Equation $I_{14}=0$ becomes $-l_{10}(l_7+l_{12})=0$,
so $l_{10}=0$, since the second factor is $2$ modulo $J$.
Now $I_{11}=0$ gives $l_5=0$, and $I_{15}=0$ gives
$l_7=l_{12}$.
Using \eqref{eq:long-elimination-product}, equation $I_3=0$
reduces to $1-l_7^2=0$. Thus $l_7=l_{12}=1$.

Finally, $K_{11}=0$ reduces to
\[
 1+l_6^2-l_8^2+l_8l_{11}-l_{11}^2
 =1-(l_8-l_{11})^2=0.
\]
Since $1+l_8-l_{11}$ is a unit, $l_8-l_{11}=1$.
Set $p=-l_{11}$ and $z=-l_6$.
Equation~\eqref{eq:long-elimination-product} becomes
$p^2-p-z^2=0$, giving the displayed matrix.

Conversely, direct multiplication gives
\[
 X_l(p,z)A(-z,p,0,1)=A(-z,p,0,1)U_l.
\]
The correcting matrix is invertible because $p^2-z^2=p$ is a unit.
It centralizes the Tits group, so conjugation transports all
standard long-root relations to $X_l(p,z)$.
\end{proof}

\subsection{The short-root certificate}\label{app:certificate-short}

Set $X_l=U_l$ and let $\mathcal G^0$ be the corresponding
specialization of $\mathcal G$.
Select the following entries as $F_1,\ldots,F_{23}$:
\[
 \begin{array}{r|c@{\qquad}r|c}
 \toprule
 i&F_i&i&F_i\\
 \midrule
 1 &(\mathcal I_s)_{3,3}
 &13&-(\mathcal K_s)_{8,8}\\
 2 &-(\mathcal I_s)_{4,4}
 &14&-(\mathcal K_s)_{7,7}\\
 3 &(\mathcal I_s)_{7,7}
 &15&-(\mathcal K_s)_{2,2}\\
 4 &-(\mathcal I_s)_{8,8}
 &16&-(\mathcal K_s)_{3,8}\\
 5 &-(\mathcal I_s)_{1,1}
 &17&(\mathcal K_s)_{2,5}\\
 6 &(\mathcal I_s)_{2,2}
 &18&(\mathcal G^0)_{3,7}\\
 7 &-(\mathcal I_s)_{3,4}
 &19&(\mathcal G^0)_{3,6}\\
 8 &-(\mathcal I_s)_{7,3}
 &20&(\mathcal G^0)_{3,2}\\
 9 &-(\mathcal I_s)_{3,7}
 &21&(\mathcal G^0)_{7,8}\\
 10&-(\mathcal I_s)_{3,8}
 &22&(\mathcal G^0)_{2,4}\\
 11&-(\mathcal I_s)_{2,6}
 &23&(\mathcal G^0)_{8,6}\\
 12&-(\mathcal I_s)_{1,5}
 &&\\
 \bottomrule
 \end{array}
\]

\label{app:certificate-minor}
The standard tuple is
\[
 \mathbf s^0=(1,1,-1,1,0,1,-1,0,0,0,1,0,
              0,0,-1,1,0,0,1,0,0,0,-1,1).
\]
Each $F_i$ vanishes there.
With rows in the displayed order and columns
$(s_1,s_3,s_4,\ldots,s_{24})$, exact differentiation and
elimination give
\begin{equation}
 \left.
 \det\frac{\partial(F_1,\ldots,F_{23})}
          {\partial(s_1,s_3,s_4,\ldots,s_{24})}
 \right|_{\mathbf s^0}
 =+384=2^7\cdot3.
 \label{eq:certificate-determinant}
\end{equation}
This determinant is a unit over $\mathbb Z[1/6]$.
The matrices, retained words, and component lists fully specify
the calculation. 

The computational supplement~\cite{TrialitySupplement} contains
the expanded polynomials, the integer Jacobian, and a script
reconstructing the entries and checking the determinant.

%---------------------------
%---------------------------
%---------------------------

\medskip
\section{Extraction of root factors}\label{app:extraction}

Let $(R,\rho)$ be a local coefficient pair and $H\trianglelefteq E_R$.
We prove Lemma~\ref{lem:extraction}.
Write $a(t)=y_a(t)$, $a_-(t)=y_{-a}(t)$, and similarly for the
other roots; $U_\Sigma$ denotes the unipotent subgroup on a closed
root set $\Sigma$. All commutators are left commutators.
Weyl transport moves an extracted parameter to any root of the same
length, up to sign. To delete an internal factor, use normality:
if $PxQ,x\in H$, with $P,Q\in E_R$, then
\begin{equation}
 (Px^{-1}P^{-1})(PxQ)=PQ\in H.
 \label{eq:extraction-delete}
\end{equation}
The extraction patterns are related to
\cite[Lemma~9.1, Case~C, and Lemma~9.2, Case~D]{GM};
we give the required identities and normality arguments below.

\subsection{Positive factors}\label{app:extraction-three-four}

For $D=d(u)e(v)f(w)\in H$, the identities
\[
 [a(1),D]_L=f(v),\qquad [a_-(1),D]_L=e(w),
\]
followed by long-root transport and deletion, extract all three
factors.

Now let $x=c(t)d(u)e(v)f(w)\in H$ and put
\[
 z(r)=[a_-(r),x]_L,\qquad
 z(r,s)=[a(s),z(r)]_L,\qquad r,s\in R_0.
\]
The signed root relations give
\begin{align*}
 z(1)&=b(t)d(-t^\#)e(w+N(t))f(-N(t)),\\
 Z_-:=z(-1,1/2)
 &=c(-t/2)d(t^\#/2)e(-N(t)/2)f(-w/2+3N(t)/4),\\
 Z_+:=z(1,-1/2)
 &=c(-t/2)d(-t^\#/2)e(-N(t)/2)f(-w/2-3N(t)/4).
\end{align*}
These polynomial matrix identities hold over
$\mathbb Z[1/2][t_0,t_1,t_2,u_0,u_1,u_2,v,w]$ before specializing
the triples to the three coefficient slots; neither $u$ nor $v$
is set to zero. Adjacent cancellation of the $c$-factors gives
\begin{equation}
 Z_-^{-1}Z_+=d(-t^\#)f(-3N(t)/2)\in H.
 \label{eq:extraction-norm-difference}
\end{equation}
Three-factor extraction gives $d(t^\#)\in H$, and
\[
 [d(t^\#),b(-t/3)]_L=e(N(t))
\]
gives both long norm factors by Weyl transport. Hence
\[
 z(1)f(N(t))e(-N(t))d(t^\#)=b(t)e(w)\in H.
\]
Now $[f_-(1),b(t)e(w)]_L=a_-(-w)\in H$.
Transport and deletion extract $e(w)$, then $b(t)$ and $c(t)$.
Deleting $c(t)$ from $x$ leaves the three-factor case.
This proves extraction for $c,d,e,f$ in that order.

For the canonical positive word
\[
 z=b(t_1)a(t_2)c(t_3)d(t_4)e(t_5)f(t_6)\in H,
\]
one has
\[
 [a(1),z]_L
 =c(t_1)d(t_1^\#)e(N(t_1))f(t_5-N(t_1)).
\]
The four-factor result and Weyl transport extract $b(t_1)$.
Delete it and denote the remaining word by $z_0$.
Modulo the normal subgroup $U_{d,e,f}$ of the positive group,
the $c$-direction is central, so
\[
 [b(1),z_0]_L=c(-t_2)d(*)e(*)f(*).
\]
Four-factor extraction and transport give $d(t_2)\in H$.
Since $t_2\in R_0$,
\[
 [d(t_2),b(-1/3)]_L=e(t_2).
\]
Transport and deletion extract $a(t_2)$, leaving the four-factor
case.

For an arbitrary order of distinct positive roots, collect the word.
Every commutator correction has greater height than either input,
so the original factors of least height retain their parameters.
Canonical extraction places these factors in $H$; delete them by
\eqref{eq:extraction-delete} and repeat in increasing height.
This extracts the original factors without identifying changed
higher canonical coordinates with their original parameters.
\label{app:extraction-positive-orders}

\subsection{Reduction of a big-cell element}
\label{app:extraction-narrow-cell}

Let $z=a(t)XhY\in H$, where $t\in R_0$,
$X\in U_0=U_{b,c,d,e,f}$, $h\in T_Q(R)^{\tau\rho}$,
and $Y\in U^-$. Set
\[
 V=U_{d,e,f},\qquad
 V^-=U_{-a,-c,-d,-e,-f},\qquad
 W^-=U_{-a,-c,-d,-f}.
\]
The long rank-one group on $\pm a$ normalizes $U_0$ and $V$,
with $V\trianglelefteq U_0$ and $[U_0,U_0]\subseteq V$.
On the negative side,
\[
 [a_-(R_0),U^-]\subseteq U_{-c,-d,-e,-f},
\]
and $W^-\trianglelefteq V^-$ centralizes $a_-(R_0)$.
The possible commutator supports with $b$ are
\[
 \begin{array}{c|ccccc}
 \text{second root}&-a&-c&-d&-e&-f\\\hline
 \text{support}
 &\varnothing&-a&-c,-a,-f&-d,-c,-a,-f&\varnothing.
 \end{array}
\]
Thus $[b(R),V^-]\subseteq W^-$, also for products by normality.
These assertions follow from the root strings and signed pair
relations.

Put $P=a(t)X$, $A_-=a_-(1)$, $B=b(1)$, and
$u=[a(-t),A_-]_L$. Expanding gives
\[
 \begin{aligned}
 P^{-1}[A_-,z]_LP
 &=[P^{-1},A_-]_L[A_-,h]_Lh[A_-,Y]_Lh^{-1},\\
 [P^{-1},A_-]_L
 &=X^{-1}uX[X^{-1},A_-]_L.
 \end{aligned}
\]
Since $u$ normalizes $U_0$, the second line is $uX_1$ for some
$X_1\in U_0$. The other factors in the first line belong to $V^-$.
Hence $z_1=uX_1Y_1\in H$ for some $Y_1\in V^-$.
This uses only the torus action on roots, not normality of $H$
under $h$.

By normality,
\[
 z_2=(uX_1A_-)^{-1}[B,z_1]_L(uX_1A_-)\in H,
\]
and
\begin{equation}
 z_2=A_-^{-1}X_1^{-1}[u^{-1},B]_LX_1
              [X_1^{-1},B]_L[B,Y_1]_LA_-.
 \label{eq:extraction-second-commutator}
\end{equation}
The element $K=[u^{-1},B]_L$ lies in $U_0$, and
\[
 X_1^{-1}KX_1[X_1^{-1},B]_L\equiv K\pmod V.
\]
Conjugation by $A_-^{-1}$ preserves this congruence, while
$[B,Y_1]_L\in W^-$ commutes with $A_-$.
Thus $z_2$ has a factorization in $U_0W^-$ whose positive factor
has the same $b,c$ coordinates as
$A_-^{-1}[u^{-1},B]_LA_-$.
The exact identity
\[
 a_-(-1)[u^{-1},b(1)]_La_-(1)
 =b(-t)c(t^2)d(t^2+t^3)e(t^2+t^3+t^4)f(-t^4-2t^5)
\]
is integral in the fixed parameter $t$.
In particular, its $b$-parameter is $-t$. We have obtained
\begin{equation}
 r=b(-t)c(p)d(q_0)e(r_0)f(s)
       a_-(v_1)c_-(v_2)d_-(v_3)f_-(v_4)\in H,
 \label{eq:extraction-r}
\end{equation}
where $p,q_0,v_2,v_3\in R$ and $r_0,s,v_1,v_4\in R_0$.

\subsection{Removing the negative factors}

Put $w_\gamma=y_\gamma(1)y_{-\gamma}(-1)y_\gamma(1)$.
We use the following operation twice.
If $Z=Ld_-(u_0)f_-(v_0)M\in H$, where $L,M\in E_R$
commute with $e(1)$, then
\[
 L^{-1}[e(1),Z]_LL
 =[e(1),d_-(u_0)f_-(v_0)]_L\in H.
\]
The universal signed identity is
\begin{equation}
 w_f[e(1),d_-(u_0)f_-(v_0)]_Lw_f^{-1}
 =b(u_0)d(-u_0^\#)e(N(u_0)-v_0)f(-N(u_0)).
 \label{eq:extraction-cleanup}
\end{equation}
It follows by multiplying the coherent matrices with independent
coefficient slots for $u_0$.
Positive extraction gives $b(u_0)$ and both long factors in $H$.
Weyl transport gives $d_-(u_0)$ and $e(N(u_0))$; root additivity
then gives $e(v_0)$ and hence $f_-(v_0)$.
Thus both original factors are extracted without requiring
$L,M\in H$.

Apply this operation to \eqref{eq:extraction-r}, with
\[
 L=b(-t)c(p)d(q_0)e(r_0)f(s)a_-(v_1)c_-(v_2),\qquad M=1.
\]
Deleting the last two factors leaves
\[
 r_1=b(-t)c(p)d(q_0)e(r_0)f(s)a_-(v_1)c_-(v_2)\in H.
\]
Its signed Weyl conjugate is
\[
 w_fr_1w_f^{-1}
 =b(-t)d_-(-p)c_-(q_0)a_-(r_0)f_-(-s)e(-v_1)d(-v_2).
\]
Among these negative factors, $a_-$ commutes with the others
and $f_-$ is central. The only nontrivial interchange is
\[
 [d_-(-p),c_-(q_0)]_L=f_-(-T(pq_0)).
\]
Consequently,
\[
 w_fr_1w_f^{-1}
 =b(-t)a_-(r_0)c_-(q_0)d_-(-p)f_-(\zeta)e(-v_1)d(-v_2),
 \qquad \zeta=-s-T(pq_0).
\]
Apply the operation again with
\[
 L=b(-t)a_-(r_0)c_-(q_0),\qquad M=e(-v_1)d(-v_2).
\]
It extracts $d_-(-p)=w_fc(p)w_f^{-1}$, so $c(p)\in H$.
Delete this original factor from $r_1$ by
\eqref{eq:extraction-delete}, obtaining
\[
 r_2=b(-t)d(q_0)e(r_0)f(s)a_-(v_1)c_-(v_2)\in H.
\]
Finally, $W=w_aw_c$ sends the successive roots
$b,d,e,f,-a,-c$ to $d,c,f,a,e,b$, with signed identity
\begin{equation}
 Wr_2W^{-1}
 =d(-t)c(-q_0)f(r_0)a(s)e(-v_1)b(v_2)\in H.
 \label{eq:extraction-final-positive}
\end{equation}
All six roots are positive and distinct.
Arbitrary-order extraction gives the original factor $d(-t)\in H$.
Since $t\in R_0$,
\[
 [d(-t),b(1/3)]_L=e(t),\qquad
 N_se(t)N_s^{-1}=a(t).
\]
This proves the big-cell assertion.
Only the fixed units $2$ and $3$ were divided by;
if $t$ is a unit, Lemma~\ref{lem:unit-generation} gives $H=E_R$.

\medskip

%----------------
%-----------------
%------------------
\section*{Acknowledgements}

This paper grew out of the second author's master's thesis at
Bar-Ilan University, supervised by the first author.

\smallskip\noindent\emph{Use of computational and AI-assisted tools.}
AI-assisted tools were used to discuss proof arguments, revise the
exposition, and prepare verification code. Selected matrix identities
and polynomial calculations were checked by exact symbolic computation.

\end{document}